\documentclass[10pt,a4paper]{article}
\usepackage{amssymb,amsmath, amsfonts}
\usepackage{stix}
\usepackage{graphicx,graphics}
\usepackage{mathtools}
\usepackage[english]{babel} 
\usepackage[utf8]{inputenc}
\usepackage{csquotes}
\usepackage{epsfig,url} 
\usepackage{bbm}
\usepackage{theorem} 
\usepackage{a4wide}
\usepackage{enumerate}
\usepackage{verbatim} 
\usepackage{color}
\usepackage{esint}

\usepackage[T1]{fontenc} 

 \usepackage[colorlinks=true, pdfstartview=FitV, linkcolor=blue, citecolor=blue, urlcolor=blue,pagebackref=false]{hyperref}
 
\newcommand{\ud}{\;\mathrm{d}}

\providecommand{\Id}{\operatorname{Id}}

\providecommand{\eps}{\varepsilon}

\providecommand{\vT}{\langle \tfrac{ v  }{\sqrt T}\rangle }

\newtheorem{theorem}{Theorem}[section]

\newtheorem{lemma}[theorem]{Lemma}

{\theorembodyfont{\upshape}
\newtheorem{remark}[theorem]{Remark}

}
\numberwithin{equation}{section}
\numberwithin{theorem}{section}
\newcommand{\qed}{\hfill$\Box$}
\newenvironment{proof}{\begin{trivlist}\item[]{\em Proof:}\/}{\qed\end{trivlist}}

\newcommand{\Reals}{{\mathbb R}}

\newcommand{\M}{{\mathcal M}}

\DeclareMathOperator{\Tr}{Tr}

\mathtoolsset{showonlyrefs}
 
\title{A blow-down mechanism for the Landau-Coulomb equation}

\date{\today}

\newcommand{\email}[1]{E-mail: \tt #1}
\newcommand{\emailMaria}{\email{gualdani@math.utexas.edu}}
\newcommand{\emailraphael}{\email{winterr6@cardiff.ac.uk}}

\newcommand{\UTaddress}{\em The University of Texas at Austin, Department of Mathematics, USA}
\newcommand{\CUaddress}{\em Cardiff University, School of Mathematics, UK}

\author{ 
 Maria Pia Gualdani \thanks{{\emailMaria.} MPG is partially supported by NSF Grant DMS-2206677. The authors would like to thank the Isaac Newton Institute in Cambridge UK for their kind hospitality during the thematic program in Spring 2022.}, Raphael Winter \thanks{\emailraphael} \\[1em]
 	$\,^*$\UTaddress \\[0.5em]
 	$\,^\dag$\CUaddress}

\date{\today}

\begin{document} 

\maketitle


\begin{abstract} 
We investigate the Landau-Coulomb equation and show an explicit blow-down mechanism for a family of initial data that are small-scale, supercritical perturbations of a Maxwellian function.  We establish global well-posedness and show that the initial bump region will disappear in a time of order one. We prove that the function remains close to an explicit function during the blow-down. As a consequence, our result shows exponential decay in time of the solution towards equilibrium. The key ingredients of our proof are the explicit blow-down function and a novel two-scale linearization in appropriate time-dependent spaces that yields uniform estimates in the perturbation parameter.

\end{abstract}

\section{Introduction} 
In this paper, we show a regularization mechanism for the spatially homogeneous Landau-Coulomb equation 
\begin{equation} \label{eq:LCEq}
	\partial_t F  = \nabla \cdot (A[F] \nabla F- \nabla a[F]F), \quad F(0,v)= F_{in}(v),
\end{equation}	
	where $v\in \Reals^3$, $A[F]$ is the anisotropic diffusion matrix, defined as 
	$$
		A[F](v)	:= \frac{1}{8 \pi}\int_{\Reals^3} \frac{\Pi(v-v')}{|v-v'|} F(v') \ud{v'},  \quad \quad  \Pi(z) := \mathbb{Id} - \frac{z \otimes z}{|z|^2},
		$$
and $a[f]$ is the trace of $A[F]$. The equation can also be written in non-divergence form as
\begin{align*}
	\partial_t F  =A[F] : \nabla^2 F+ F^2, \quad F(0,v)= F_{in}(v).
\end{align*}
A crucial characteristic of this equation is that the diffusion term $A[F]: \nabla^2 F$ and the drift $F^2$ are competing terms with the same scaling properties. This fact has made it challenging to settle the question of global well-posedness versus blow-up in finite time. 
Since the groundbreaking result of Guillen and Silvestre~\cite{guillen2023landau}, we now know that solutions to~\eqref{eq:LCEq} do not blow up if the initial data are reasonable.  

The mechanisms by which the Landau equation eliminates small-scale singularities are not yet fully understood. Prior to \cite{guillen2023landau}, a major challenge in proving global well-posedness was the absence of mathematical tools to control perturbations with small mass, which could potentially grow undetected by the available physical bounds: mass, momentum, energy and entropy. Guillen and Silvestre in \cite{guillen2023landau} identified the Fisher information as a Lyapunov functional of the equation, which provides a new a priori bound. This bound allows for global well-posedness if the initial data have bounded Fisher information or if the initial data are in $L^{p}$ with $p\ge \frac{3}{2}$ \cite{Golding_Gualdani_Loher}. Currently, global well-posedness for \eqref{eq:LCEq} with initial data in $L^p$ with $p<3/2$ remains an open problem. In this regime, we only know partial space-time regularity estimates \cite{PartialRegularity2,PartialRegularity1}.

Given that blow-up does not occur if the initial data are reasonably well-behaved, it is interesting to explore the blow-down mechanism. The goal of this work is a first step towards the theory of supercritical initial data for the Landau-Coulomb equation, and an understanding of their regularization and decay to equilibrium. The current theory does not yet cover initial data with supercritical singularities but strong decay at infinity. Here, we define as {\em{subcritical}} all $L^p$-norms with $p>\frac{3}{2}$, as {\em{critical}} the $L^{3/2}$-norm and  {\em{supercritical}} all $L^p$-norms with $p<\frac{3}{2}$. These terms are not related to the equation's scaling properties, but rather to the fact that in the subcritical regime the equation \eqref{eq:LCEq} has bounded coefficients, whereas this boundedness might fail in the critical and subcritical regime.

We begin by examining initial data with singularities that possess large super- and sub-critical norms and large Fisher information. We consider as initial data small-scale perturbations of a Maxwellian function 
\begin{align}\label{initial_data}
	F_{in}(v) = \M(v) +  \delta^{\alpha} \M_{\delta}(v),
\end{align}
where  the Maxwellian functions are given by
\begin{align} \label{eq:Maxwellians} 
\M(v) = \frac{1}{(2\pi)^{3/2}} e^{-\frac12|v|^2}, \quad  \M_\delta(v) = \frac{1}{(2\pi \delta)^{3/2}} e^{-\frac{|v|^2}{2\delta}}.
\end{align}
We restrict to the regime $0<\delta \ll 1$ and $\alpha\in (\frac14, \frac12)$. For these values of $\alpha$, the perturbation $\delta^{\alpha} \M_{\delta}$ has small $L^1$-norm but large $L^p$-norms for all $p>p_0$ with $p_0<\frac32$ depending on $\alpha$. Hence, the perturbation is large even in spaces below the critical integrability ($p=\frac{3}{2}$ for Landau-Coulomb).  In particular, the $L^\infty$-norm and the Fisher information of $F_{in}$ are diverging as $\delta \to 0$.  

We briefly recall that global well-posedness for small perturbation of the steady state  in $L^p$-norms with $p > \frac{3}{2}$ has been well-known since some time \cite{carrapatoso2017landau,guo_landau_2002, KGH,DBJ,GGL}.

Our analysis reveals the exact blow-down mechanism: on a time interval $(0, t^\ast)$ the solution to \eqref{eq:LCEq}, \eqref{initial_data}  is the sum of three components: a fixed Maxwellian function, the local singularity (emerging from the   $ \delta^{\alpha} \M_{\delta}(v)$ part of the initial data) and  a small perturbation that disappears as time grows. More precisely, we show that 
\begin{equation*}
	F(v,t) = \M(v) + E(v,t) + f(t,v),
\end{equation*}
where $f$ is some small perturbation, and $E(t,v)$ is an explicit function given by ($c_0$ given in~\eqref{def:c0})
\begin{align}\label{ETm}
E(t,v) = m(t) \M_{T(t)}, \quad T(t) := \delta + 2 c_0 t, \quad m(t) = \delta^\alpha e^{\frac{2}{(2\pi)^{3/2}}t}.
\end{align}
Let us explain what we mean by blow-down mechanism: the main players are $\M$ (the unit Maxwellian) and $E$ (the explicit local singularity), while the perturbation $f$ is the solution to a  nonlinear problem with zero initial data. We show that $f$ remains small in the time interval $(0, t^\ast)$.  The local singularity $E$, which at time $t=0$ has small mass  but is large in supercritical $L^p$-norms, regularizes as soon as $t>0$ and gains integrability as time increases, {\em{uniformly}} with respect to $\delta$. Note in fact that 
$$
\|E(t,\cdot)\|_{L^p}\sim  \frac{\delta^{\alpha}}{(\delta + 2c_0 t)^{\frac32 (1-\frac1{p})}}.
$$Crucially, our argument shows that at a certain time $t^\ast>0$, our solution $F$ is close to $\M$ and satisfies  
\begin{align*}
\|F(t^\ast,\cdot) - \M\|_{L^2_{\M}} \le o(1), \quad \text{for $\delta \rightarrow 0$}. 
\end{align*}
The time $t^\ast>0$ can be chosen uniform in $\delta \rightarrow 0$. This shows that the Landau collision operator has blown-down the perturbation $E+f$ and at time $t^\ast$ the solution $F$ is close to the unit Maxwellian. Since $\M$ is close to the global equilibrium $\M_{eq}$ associated to the initial data $F_{in}$
\begin{align*}
\M_{eq}(v) =(1+\delta^\alpha ) \M_{\frac{1+\delta^{1+\alpha}}{1+\delta^\alpha}}(v),
\end{align*}
at time $t^\ast>0$ our solution $F$ satisfies
\begin{align*}
\|F(t^\ast,\cdot) - \M_{eq}(\cdot) \|_{L^2_{\M}} \le o(1), \quad \text{for $\delta \rightarrow 0$}. 
\end{align*}
After $t^\ast$, the distance between $F$ and the equilibrium $\M_{eq}$ decays exponentially fast. This follows from existing results, for example~\cite{carrapatoso2017landau}.

Summarizing, our result is threefold: 
\begin{enumerate}
	\item we prove blow-down mechanism for a family of  supercritical perturbations, i.e. perturbations that are near Maxwellians only in $L^p$ norms with $p$ close to $1$. 
	\item we show that during the blowdown time the solution remains close to an explicitly computable function, namely  $\M(v) + E(v)$, with $E$ defined in~ \eqref{ETm}, 
	\item we prove exponential convergence to equilibrium, uniform in the perturbation parameter $\delta>0$. 
\end{enumerate} 

The main result of our manuscript is the following theorem. 

\begin{theorem} \label{thm:main} 
For $\alpha \in (\frac14,\frac12)$ there exists $\delta_0 \in (0,\frac12)$ such that for $\delta \in (0,\delta_0)$ the following holds: Let $F_{in}$ be as defined as in~ \eqref{initial_data}. Then there exists a unique global smooth solution $F$ to \eqref{eq:LCEq}, which has the form 
$$
F(t,v) = \M(v) + E(t,v) + f(t,v),
$$
with $E$ defined in \eqref{ETm}. At  $t=t^*$, where $t^*>0$ is some time of order one, specifically ($c_0$ given in~\eqref{def:c0})
 $$
 t^* \le \min \left \{ \frac{1}{2}, \frac{1}{4c_0} \right\} =\frac12 ,
 $$
 the perturbation is small in weighted $L^2$ norm:  
$$
\| E + f \|_{L^2_{\M}} \le C \delta^{\alpha-\frac14} ,
$$
and, for some $\lambda_0>0$, decays as
\begin{align*}
\|F(t,\cdot) - \M_{eq} \|_{L^2} \leq C \delta^{\alpha-\frac14} e^{-\lambda_0 t^{2/3}}, \quad t>t^*.
\end{align*}
Here $\M_{eq}$ is the unique equilibrium function with the same mass and second momentum as $F_{in}$: 
$$\M_{eq} =(1+\delta^\alpha ) \M_{\frac{1+\delta^{1+\alpha}}{1+\delta^\alpha}}.$$
\end{theorem}
\begin{remark}
	In Theorem~\ref{thm:main} as well as the rest of the paper, we will denote by $C,c>0$ large respectively small constants which are independent of $\delta$, $t$, $v$ and may change from line to line. 
\end{remark}

Several tools play a key role in our result. We do not rely on the bound of the Fisher information (which is indeed diverging as $\delta \to0$) but introduce a new type of linearization that involves $\M$, as well as the explicit local singularity $E(t)$. The local blow-down profile represent the evolution of a small Maxwellian and solves the linear inhomogeneous heat equation \eqref{eq:E}. Hence, we can compute $E(t)$ explicitly.  The crucial technical point is the functional setting: we work with time-dependent Banach spaces that can measure explicitly the contribution coming from the large scale and the contribution from the time-dependent small scale. We call these two contributions respectively far-field and near-field. We split the perturbation $f$ into its far-field and near-field component and rewrite the Landau equation as a coupled two-scale system (see \eqref{eq:hgcoupled}).
We require the assumption $\alpha>\frac14$ in order to be able to control the near-field perturbation such as the right-hand side in~\eqref{Q(g,E)}.

After establishing the functional framework, the estimates in the near-field and in the far-field rely on the coercivity provided by the large scale Maxwellian $\M$, and on the monotonicity of the eigenvalues of $A[\M]$ and other related quantities.  

Although we prove our theorem for the Landau-Coulomb case, similar results should apply for the Landau operator in the soft potential case, and for the Boltzmann equation. Moreover, one could inquire whether a similar blow-down mechanism applies to the inhomogeneous equation. It is noteworthy that the initial data considered here will lead to nearly self-similar blow-up for a modified Landau equation (cf.~\cite{Chen}):
$$
\partial_t F =  \nabla \cdot (A[F] \nabla F- \nabla a[F]F) + \varepsilon F^2, \quad \varepsilon >0.
$$
The choice of initial data~\eqref{initial_data} in our work was motivated by the considerations in~\cite{Silvestre2023} (see p. 19).
 In a similar spirit, we recall a recent result for the compressible Navier-Stokes equation in 3-dimensional space  (see~\cite{Bing_Xu_Zhang}) with pulse initial data. The authors show that the initial small-scale bump will disappear in a short time, leading to global well-posedness of the solution which remains small in energy norms.

\section{Setup of the problem and outline of the proof}	
\subsection{Two-scale linearization}

We consider initial data of the form:
\begin{align*}
	F_{in} = \M + \delta^{\alpha} \M_{\delta},
\end{align*}
where $\M$ is the Maxwellian with temperature $1$ and $ \M_{\delta}$ the one of temperature $\delta$ defined in~\eqref{eq:Maxwellians}.
The parameter $\alpha$ is positive and strictly less than $\frac{1}{2}$.  The threshold  $\frac{1}{2}$ for $\alpha$ implies that the $L^p$-norms with $p<\frac{3}{2}$ of $F_{in}$ increase as $\delta \to 0$. Note in fact that 
\begin{align*}
	\|F_{in}\|_{L^p} = C(p,\M) + C(p) \delta^{\alpha-\frac32 (1-\frac1{p})}. 
\end{align*}

Let us introduce a new linearization which combines the "near-field" and "far field" regime. More precisely, we look for solutions of the form
\begin{align*}
	F(t,v) 	&= G(t,v)+ H(t,v),
	\end{align*}
	with $G$ the far-field  and $H$ the near-field.  Here, both $G$ and $H$ are further decomposed into
	\begin{align*}
	G(t,v)	&= \M(v)+ g(t,v),\\
	H(t,v)	&= E(t,v)+ h(t,v).
\end{align*}
The explicit function $E$ (cf.~\eqref{ETm}) carries the leading-order time evolution of the perturbation $\delta^\alpha \M_\delta$ and is determined by the linear problem
\begin{align}\label{eq:E}
	\partial_t E 	&= c_0 \Delta E + 2 \M(0) E, \\
		E(0,v)		&= \delta^\alpha \M_\delta(v),	 \nonumber
\end{align}
where the constant $c_0$ is given by
\begin{align} \label{def:c0}
	c_0 := \sup_{\|e\|=1}\langle A[\M](0)e,e\rangle = (3 (2\pi)^\frac32)^{-1} .
\end{align}
We note that the problem \eqref{eq:E} has the explicit solution 
\begin{align*}
	E(t,v) 	=  \delta^\alpha e^{2\M(0)t}  \M_\delta \ast \M_{2 c_0 t}  =  \delta^\alpha e^{2\M(0)t} \M_{T(t)}   ,
	\end{align*} 
	with $T(t)$ the sum of the variances of both Maxwellians: 
	\begin{equation} \label{def:T}
		T(t)	= \delta +  2c_0t.
	\end{equation}		
While $E$ solves~\eqref{eq:E}, the function $g+h$ is zero at time $t=0$ and solves for $t>0$ 
\begin{align*}
(h+g)_t = Q(\M+g+E+h) - c_0\Delta E - 2\M(0)E,
\end{align*}
where $Q$ is the Landau-Coulomb collision operator.
We decompose the above problem into a system of coupled nonlinear equations for a small-scale perturbation $h$ and a large-scale perturbation $g$. To this end, we first introduce the bilinear operator
\begin{align*}
	Q(F,G) = \nabla \cdot (A[F] \nabla G - \nabla a[F] G).
\end{align*}
Then we introduce the two-scale coupled system
\begin{equation} \label{eq:hgcoupled} 
\begin{aligned}
\partial_t h = \; &Q(h,h) + Q(g, h) + Q(E,h)+Q(h,E) +Q(\M,h)+Q(g,E)\\
&\; + gh+h\M+gE+\M E + Q(\M,E) - c_0\Delta E - 2\M(0)E, \\
 h(0,&v)=0,\\
\partial_t g = \; &Q(g,g)+Q(\M,g)+Q(g,\M)\\
&+ \Tr(A[h]\nabla^2\M) +  \Tr(A[h]\nabla^2g)+\Tr(A[E]\nabla^2g)+\Tr(A[E]\nabla^2\M), \\
g(0,&v)=0.
\end{aligned}
\end{equation}
Additional crucial ingredients are the function spaces 
\begin{align} 
	\|h\|^2_{L^2_{\mu_{T(t)}}} 	&:= \int_{\Reals^3} h^2 \mu^{-1}_{T(t)} \ud{v}, \quad \text{for $t\geq 0$,} \label{def:L2Mu} \\
	\|g\|^2_{L^2_\M}	&:= \int_{\Reals^3} g^2 \M^{-1} \ud{v},  \label{def:L2M}
\end{align}
where $T(t)$ is given by~\eqref{def:T} and $\mu_{T(t)}$ is the time-dependent weight:
\begin{align} \label{def:mu} 
	\mu_{T(t)} = {T(t)}^\frac32 e^{-\frac{|v|^2}{2{T(t)}}}.
\end{align}
Note that for a time of order one, $E(t)$ and $\mu_{T(t)}$ are related by the following relation: 
$$
E(t,v) \approx \frac{\delta^\alpha}{T^3} \mu_{T(t)}, \quad T(t)	= \delta +  2c_0t. 
$$
We will often suppress the dependence on $t$ in $T(t)$ and just write $\mu$ or $\mu_{T} $ instead of $\mu_{T(t)} $ when the suppression will not be misleading.
 
 The next theorem establishes the existence of $g$ and $h$. For convenience, we will leverage the existence results from \cite{guillen2023landau} to demonstrate the existence of $g$ and $h$,  thereby avoiding unnecessary lengthy computations. However, we emphasize that the a priori estimates presented later would be more than sufficient for an independent proof using an appropriate fixed-point argument.

\begin{theorem} \label{thm:existence}
The system~\eqref{eq:hgcoupled} is well-posed, and the functions $g$ and $h$ are smooth and radially symmetric. Moreover, for $\delta \in (0,\frac12)$ we have  $g\in C^1([0,1];L^2_\M)$ and $\|h(t,\cdot)\|_{L^2_\mu} \in C^1([0,1])$.  
\end{theorem}
\begin{proof}
Let us denote, for simplicity of notation, the linearized operators 
\begin{align}
	L_E(h)	&:= Q(E,h)+ Q(h,E), \label{def:LE}\\
	L_\M(g)	&:= Q(\M,g)+ Q(g,\M) \label{def:LM},
\end{align}
and with $S_E$ and $S_\M$ the terms 
\begin{align}
	S_E &:= \M E + Q(\M,E) - c_0\Delta E - 2\M(0)E, \label{def:SE}  \\
	S_\M &:= \Tr(A[E]\nabla^2\M). \label{def:SM} 
\end{align}
With this simplified notation, the equations for $g$ and $h$ become
\begin{equation*} 
\begin{aligned}
	\partial_t g &= L_\M(g) + A[h]:\nabla^2 G + A[E]:\nabla^2 g + Q(g,g)+ S_\M,\\
	\partial_t h 
	& = L_E(h) + h(\M+g) +Eg + Q(\M,h) + Q(g,h) + Q(h,h)+ Q(g,E) + S_E,
\end{aligned}
\end{equation*}
The existence of the unique smooth, decaying, radial solution $F$ is guaranteed by Theorem 1.1 in~\cite{guillen2023landau}.  
We also remark that the  equation for $h$ can be rewritten as
\begin{align*}
	\partial_t h 	
					&=  Q(F,h)+Q(F,E)+ h F -2 h E - h^2 + EF  - E \M -E^2  - Q(\M,E) +S_E.
\end{align*}
The equation for $h$ in this form is a linear second order diffusion equation up to the dissipative term $-h^2$, so existence of a smooth solution with $\|h(t,\cdot)\|_{L^2_\mu} \in C^1([0,1])$ is straightforward. The existence of $g$ then follows from
\begin{align*}
	g = F- E-\M-h.
\end{align*}
Here we use $E(t,\cdot)\in L^2_\M$ for $t\leq 1$ and $\delta \in (0,\frac12)$.  Due to uniqueness and the rotational symmetry of the initial data, all functions involved are rotationally symmetric. 
\end{proof}

 {{Note that the previous theorem only shows the existence of smooth decaying functions $g,h$, not their smallness. The control of the size of the perturbation will be the topic of the next sections.  }}

\subsection{Preliminaries}

To keep the paper self-contained, we recall some basic properties of the Landau-Coulomb equation. First we recall the following identity for the operators $A$ and $a$ in~\eqref{eq:LCEq}:
\begin{align}
	\nabla \cdot A[F] 	&= \nabla a[F] \label{eq:Aa},\\
		\Delta a[F]		&=-F. \label{eq:aDelta}
\end{align}
It follows immediately that the Landau operator can be rewritten in the well-known non-divergence form
\begin{align} \label{eq:nonDiv} 
	Q(F,G) &= \nabla \cdot (A[F] \nabla G- \nabla a[F] G) = A[F] : \nabla^2 G + FG.
\end{align}
Moreover, we recall the following formula for $a[f]$ if $f$ is radially symmetric:
\begin{align*}
	a[f](v) = \frac{1}{4\pi |v|} \int_{B_{|v|}} f(w) \ud{w} + \frac{1}{4\pi} \int_{B_{|v|}^c} \frac{1}{|w|} f(w) \ud{w}.  
\end{align*}
This quickly yields the monotonicity of $a[M]$. This observation is not novel, but we state it here for the sake of completeness.
\begin{lemma}\label{lem:eq:monotone} 
	If $f\in L^1(\Reals^3)\cap L^2(\Reals^3)$ is radially symmetric, positive, and monotonically decreasing, then $a[f](v)=\phi(|v|)$ is also radially symmetric and decreasing, i.e.
	\begin{align}\label{eq:monotone} 
		\partial_r \phi(r)= -\frac{1}{r^2} \int_0^r y^2 f(y) \ud{y} \leq 0.
	\end{align}
\end{lemma}
\begin{proof}
	Follows simply by using polar coordinates and taking the derivative:
	\begin{align*}
		\partial_r \phi(r) &= \partial_r \left(\frac{1}{r}\int_0^r y^2 f(y) \ud{y}+ \int_r^\infty y f(y)\ud{y} \right) \\
							&= -\frac{1}{r^2} \int_0^r y^2 f(y) \ud{y} \leq 0,
	\end{align*}
proving the claim since $f\geq 0$. 
\end{proof}

Moreover, we will make use of the monotonicity of the eigenvalues of $A[M]$. This property is the content of the following Lemma.

\begin{lemma} \label{lem:EVmon} 
	For any $v\in \Reals^3$, the matrix $A[M](v)$ has two positive Eigenvalues $\lambda_1(|v|)$, $\lambda_2(|v|)$ and $\lambda_1$, $\lambda_2$ are monotone decreasing in $|v|$.  
\end{lemma}
\begin{proof}
	We first observe that $v$ is an Eigenvector of $A[M](v)$, and the matrix is a multiple of the identity on the orthogonal space to $v$.  Let us denote by $\lambda_1$ the Eigenvalue corresponding to the Eigenvector $v$. Then the second Eigenvalue can be written as
	\begin{align*}
		\lambda_2(v) &= \frac12 \left( \operatorname{Trace} A[M](v)-\lambda_1(v)\right) \\
					&= \frac12 \left( a[M](v) - \lambda_1(v)\right). 
	\end{align*}
	From~\cite{gualdani_estimates_2016}~Proposition~3.2, we have the following representation of $\lambda_1(|v|)$
	\begin{align*}
		\lambda_1(\rho) &= \frac1{3\rho^3} \int_0^\rho \M(s) s^4 \ud{s} + \frac{1}{3 } \int_{\rho}^\infty \M(s) s \ud{s}. 
	\end{align*}
	We quickly verify the monotonicity of $\lambda_1$ by computing
	\begin{align*}
		\lambda_1'(\rho) &= -\frac{1}{\rho^4} \int_\rho^\infty \M(s) s^4 ds <0. 
	\end{align*}
	Moreover, we use the representation of $\lambda_2$ above to find
	\begin{align*}
		\lambda_2'(\rho) &= -\frac12 \int_\rho^\infty \big( \frac{1}{\rho^2} s^2-\frac{1}{\rho^4} s^4\big) \M(s) \ud{s}  <0,
	\end{align*}
	by positivity of the integrand. 
\end{proof}

An important tool to control the growth of the weighted $L^2$ norms for $h$ and $g$ are coercivity estimates. For the near-field $h$, we extract coercivity from the scale-separation, in particular from the term $Q(\M,h)$. For the near field $g$, we recall a classical coercivity result by Guo (Lemma~5 in~\cite{guo_landau_2002}), translated to our functional framework. First, notice that our linear operator $L_\M $ (cf.~\eqref{def:LM}) and Guo's linear operator (here denoted for simplicity by $ L_{\text{Guo}}$) are related by the following change of function:
\begin{equation} \label{eq:UsGuo} 
\begin{aligned}
	(- L_{\text{Guo}} (\psi)) \sqrt{\M} &= L_\M g, \quad  \text{where} \\
		g&= \sqrt \M \psi. 
\end{aligned}
\end{equation}
Using these identities, $L_\M$ satisfies the following coercive estimate:
\begin{lemma} \label{lem:firstcoerc} 
	There exists a constant $\bar{c}_0>0$ such that
	\begin{align}\label{eq:LM}
			-\int_{\Reals^3} L_\M(g) g \M^{-1}\ud{v} &\geq  \bar{c}_0 D_\M(P^\perp_\M g).
	\end{align}
	Here $P^\perp_\M g=(\Id - P_\M) g$  with $P_\M$ the projection onto the kernel of $L_\M$ and  $D_\M(g)$ given by
\begin{align} \label{eq:DMG} 
	D_\M(g) &= \int_{\Reals^3} \ ( \langle v\rangle^{-1}g^2+ \langle v \rangle^{-3} |\nabla g|^2) \M^{-1}\ud{v}.
\end{align}
The kernel is given by 
\begin{align} \label{eq:kernel} 
	\operatorname{ker} L_\M = \{g= p \M: p(v) = a+ b\cdot v +c |v|^2, \text{ for some $a,c\in \Reals$, $b\in \Reals^3$}\}, 
\end{align}
and the projection is taken with respect to the inner product associated to the norm $L^2_\M$ (cf.~\eqref{def:L2M})
\begin{align*}
	(f,g)_\M = \int_{\Reals^3} f(v) g(v) \M^{-1} \ud{v}. 
\end{align*}
\end{lemma}
\begin{proof}
	The proof follows directly from (Lemma 5 in~\cite{guo_landau_2002}). We recall that
	\begin{align*}
		\operatorname{ker} L_{\text{Guo}} = \{ \psi=p \sqrt \M: p(v) = a+ b\cdot v +c |v|^2, \text{ for some $a,c\in \Reals$, $b\in \Reals^3$}\},
	\end{align*}
	so~\eqref{eq:kernel} follows from the relation~\eqref{eq:UsGuo}. The projection onto the kernel can be found analogously. 
	It remains to check the coercivity estimate. Guo's estimate, in the radially symmetric case, reads 
	\begin{align*}
		\int_{\Reals^3} L_{\text{Guo}} (\psi) \psi \ud{v} &\geq c_0 D_{\text{Guo}}( P^\perp_{\sqrt \M} \psi ) , \\
		D_{Guo} (\phi) &= \int_{\Reals^3} A[\M]_{i,j} (\nabla_i \phi \nabla_j \phi + v_i v_j \phi^2)  \ud{v},
	\end{align*}
	where $P^\perp_{\sqrt \M}$ is the orthogonal projection to $\operatorname {ker} L_{\text{Guo}}$. Inserting $\psi = g \M^{-\frac12}$ and writing 
	\begin{align*}
		f :=P^\perp_\M g = \sqrt{M} P^\perp_{\sqrt \M} \psi,
	\end{align*} 
	we obtain 
	\begin{align*}
		-\int_{\Reals^3} L_\M(g) g \M^{-1}\ud{v} \geq c_0 \int_{\Reals^3} A[\M]_{i,j} ((\nabla_i +v_i)f (\nabla_j+v_j) f  + v_i v_j f^2) \M^{-1} . 
	\end{align*}
	This gives us the desired coercive terms for the gradient, plus a remainder:
	\begin{align*}
		-\int_{\Reals^3} L_\M(g) g \M^{-1}\ud{v} \geq &\;  c_0 \left(\int_{\Reals^3} A[\M]_{i,j}  \nabla_i f \nabla_j f  \M^{-1} +2 \int_{\Reals^3} A[\M] v^{\otimes 2} f^2 \M^{-1} \right) \\
		& -  c_0 \left(2|\int_{\Reals^3} A[\M]_{i,j}  v_i f \nabla_j f \M^{-1} |\right). 
	\end{align*}
	Using Young's inequality, the last term can be estimated by
	\begin{align*}
		2\left |\int_{\Reals^3} A[\M]_{i,j}  v_i f \nabla_j f \M^{-1} \right|\leq \frac23 \int_{\Reals^3} A[\M]_{i,j}  \nabla_i f\nabla_j f \M^{-1} + \frac32 \int_{\Reals^3} A[\M]_{i,j}  v_i f v_j f \M^{-1}.
	\end{align*}
	Combining this with the previous estimate we obtain
	\begin{align*}
		-\int_{\Reals^3} L_\M(g) g \mu^{-1}\ud{v} \geq \bar{c}_0 \int_{\Reals^3} A[\M]_{i,j} (\nabla_i f \nabla_j f + v_i v_j f^2) \M^{-1} \ud{v},
	\end{align*}
	as claimed. 
\end{proof}

The next lemma summarizes elementary computations on the weight $\mu_{T}$ introduced in~\eqref{def:mu}. 
\begin{lemma}
	The following identities hold: 
	\begin{align*}
		\nabla \mu_{T}^{-1} &= \frac{v}{T}\mu_{T}^{-1}, \quad \quad  \quad \nabla \mu_{T}^{-1}  = \frac{v}{|v|} | \nabla \mu_{T}^{-1} |,\\
		| \nabla \mu_{T}^{-1/2} |^2 &= \frac{|v|^2}{T^2}\mu_{T}^{-1} ,\quad \; \; \;  | \nabla \mu_{T}^{-1/4} |^2 = \frac{|v|^2}{16 T^2}\mu_{T}^{-1/2}, \\
		\partial_t \mu_{T}^{-1} &= -c_0\left( \frac{3}{T} + \frac{|v|^2}{T^2}\right) \mu_{T}^{-1} .
	\end{align*}
\end{lemma}
\begin{proof}
	The identities can be verified by straightforward computations.
\end{proof}

{

\section{Coercivity of the two-scale linearization} 

\subsection{Estimates on the near-field scale}

The classical result by Guo (Lemma~5 in~\cite{guo_landau_2002}) implies a coercivity estimate for the linearized operator $L_E$ defined in~\eqref{def:LE}. The coercivity provided by this estimate is relatively weak and will not be used in the main result. We will only make use of the non-negativity of the term. 
\begin{lemma}
	The following estimate holds for any $h\in L^2_\mu$:
	\begin{align}\label{Q(E,h)+Q(h,E)}
		\int_{\Reals^3} L_E(h) h \mu^{-1}\leq 0.
	\end{align}
\end{lemma}
\begin{proof}
	By rescaling we obtain that the statement is equivalent to
	\begin{align*}
		\int_{\Reals^3} L_\M(h) h \M^{-1}\leq 0,
	\end{align*}
	for any $h\in L^2_\M$. This is shown in Lemma~\ref{lem:firstcoerc}.
\end{proof}

We now turn to the coercivity estimate in the near-field. It will be provided by the interaction of the small-scale perturbation $h$ with the large-scale Maxwellian $\M$. In addition to the coercive term, we crucially make use of the non-positivity of the term~\eqref{eq:monotone}.

Finally, we turn to the estimate of the source term in the near-field. Note that we set up the linearization such that we gain a quadratic term for small $v$ in~\eqref{eq:quadratic} below. 
\begin{lemma} [Source in the near-field] Recall the norm $L^2_\mu$ introduced in~\eqref{def:L2Mu}.
	The contribution of the source term $S_E$ defined in~\eqref{def:SE} can be estimated by
	\begin{align}\label{SE}
		\int_{\Reals^3} h S_E \mu_{T}^{-1} \ud{v}\leq C m T^{-\frac12} \|h\|_{L^2_\mu}.
	\end{align}
\end{lemma}
\begin{proof}
	We use the symmetry and smoothness of $\M$ to infer
	\begin{align} \label{eq:quadratic}
		|A[\M](v) -A[\M](0)| + |\M(v)-\M(0)|\leq C |v|^2. 
	\end{align}
	This allows us to obtain the bound
	\begin{align*}
		\int_{\Reals^3} h S_E \mu_{T}^{-1} \ud{v} &\leq C T \int_{\Reals^3} |h| \frac{|v|^2}{T} E \mu_t^{-1} \ud{v} \\
		&\leq C T\|h\|_{L^2_\mu} \left(\int_{\Reals^3} \frac{|v|^4}{T^2}E^2 \mu^{-1} \right)^\frac12 \\
		&\leq CmT \|h\|_{L^2_\mu} \left(\frac{1}{T^\frac92}\int_{\Reals^3} \frac{|v|^4}{T^2}e^{-|v|^2/T} \right)^\frac12\\
			&\leq C m T^{-\frac12} \|h\|_{L^2_\mu} 
	\end{align*}
	as claimed. 
\end{proof}

The coercivity in the equation for $h$ is provided by the next lemma.

\begin{lemma}
	There exist constants $c,C>0$ such that the following estimate holds
	\begin{align}\label{Q(M,h)}
		{\int_{\Reals^3} Q(\M,h) h \mu^{-1} \;dv   \leq  - c\int_{\Reals^3}  \frac{| \nabla h|^2}{\langle v \rangle^3}  \mu^{-1} \;dv \; -  \frac{1}{2} \int_{\Reals^3} \partial_t  \mu^{-1} h^2  \;dv + C \|h\|^2_{L^2_\mu}}.
	\end{align}
\end{lemma}
\begin{proof}
	Integrating by parts, we write 
	\begin{align*}
		\int_{\Reals^3} Q(\M,h) h \mu^{-1} \;dv =&  - \int_{\Reals^3} \langle A[\M] \nabla h , \nabla h \rangle  \mu^{-1} \;dv -  \int_{\Reals^3} \langle A[\M] \nabla h , \nabla   \mu^{-1} \rangle h \;dv\\
		& + \int_{\Reals^3} h \nabla \cdot (A[\M])\cdot \nabla(h \mu^{-1})\;dv \\ 
		=  - &\int_{\Reals^3} \langle A[\M] \nabla h , \nabla h \rangle  \mu^{-1} \;dv -\frac{1}{2} \int_{\Reals^3}  \langle \nabla \cdot (  A[\M]  \mu^{-1}) , \nabla h^2\rangle \;dv \\
		&+ \int_{\Reals^3} h \langle \nabla h, \nabla \cdot (  A[\M]) \rangle \;dv  + \int_{\Reals^3} h \nabla \cdot (A[\M])\cdot \nabla(h \mu^{-1})\;dv \\  
		= - &\int_{\Reals^3} \langle A[\M] \nabla h , \nabla h \rangle  \mu^{-1} \;dv + \frac{1}{2} \int_{\Reals^3}  \nabla^2(  A[\M]  \mu^{-1}) h^2  \;dv + \int_{\Reals^3} h^2 \M  \mu^{-1}\;dv \\
		\le - c&\int_{\Reals^3}  \frac{| \nabla h|^2}{\langle v \rangle^3}  \mu^{-1} \;dv + \int_{\Reals^3} h^2 \M  \mu^{-1}\;dv + \frac{1}{2} \int_{\Reals^3}  \nabla^2(  A[\M]  \mu^{-1}) h^2  \;dv .
	\end{align*}
	In the last step, we have used the classical bound
	\begin{align*}
		e A[\M](v)e \geq \frac{c |e|^2}{\langle v\rangle^3 }, \quad \forall e\in \Reals^3.
	\end{align*}
	It remains to estimate the last integral. To this end, we notice that 
	\begin{align*}
	\nabla^2(  A[\M]  \mu^{-1}) = -\M  \mu^{-1} + 2 \nabla \cdot (A[\M]) \cdot \nabla  \mu^{-1} + \Tr(A[\M] \nabla^2 \mu^{-1}).
	\end{align*} 
	At this point, we recall that due to radial symmetry (see Lemma~\ref{lem:eq:monotone}) the following term is non-positive
	\begin{align*}
		\nabla \cdot (A[\M]) \cdot \nabla  \mu^{-1} = \nabla a[M] \nabla (\mu^{-1}) \leq 0,
	\end{align*}
	and furthermore we have the identity
	\begin{align*}
		\nabla  \mu^{-1} = \frac{v}{T}  \mu^{-1}, \quad \quad \Delta  \mu^{-1} = \left( \frac{3}{T} + \frac{|v|^2}{T^2}\right)  \mu^{-1}  = -\frac{1}{c_0} \partial_t  \mu^{-1}.
	\end{align*}
	This allows us to obtain the upper bound  
	\begin{align*}
		\int_{\Reals^3} Q(M,h) h \mu^{-1} \;dv \le  &  - c\int_{\Reals^3}  \frac{| \nabla h|^2}{\langle v \rangle^3}  \mu^{-1} \;dv +\frac{1}{2} \int_{\Reals^3} h^2 \M  \mu^{-1}\;dv + \frac{1}{T} \int_{\Reals^3}  \nabla \cdot (A[\M]) \cdot {v}   \mu^{-1}  h^2  \;dv \\
		&  +  \frac{1}{2} \int_{\Reals^3} \Tr(A[\M] \nabla^2 \mu^{-1}) h^2  \;dv +\frac{1}{2} \int_{\Reals^3} h^2 \M  \mu^{-1}\;dv \\
		\leq &   - c\int_{\Reals^3}  \frac{| \nabla h|^2}{\langle v \rangle^3}  \mu^{-1} \;dv {+  \frac{1}{2} \int_{\Reals^3} \Tr(A[\M] \nabla^2 \mu^{-1}) h^2  \;dv }+ C \int_{\Reals^3} h^2 \M  \mu^{-1}\;dv .
	\end{align*}
	By Lemma~\ref{lem:EVmon}, the Eigenvalues of $A[M](v)$ are decreasing. Therefore we can estimate
	\begin{align*}
		\Tr(A[\M] \nabla^2 \mu^{-1}) \leq \lambda(v) \Delta \mu^{-1},
	\end{align*}
	where $\lambda(v)$ is the largest Eigenvalue of $A[M](v)$.  Since $c_0$ in \eqref{def:c0} is the Eigenvalue of $A[M](0)$, we have $0<\lambda \leq  c_0$. The Laplacian $\Delta \mu^{-1}$ is positive, so we can bound   
	\begin{align*}
		\int_{\Reals^3} Q(M,h) h \mu^{-1} \;dv \leq &   - c\int_{\Reals^3}  \frac{| \nabla h|^2}{\langle v \rangle^3}  \mu^{-1} \;dv \; {{+  \frac{c_0}{2} \int_{\Reals^3}  \Delta \mu^{-1} h^2  \;dv }}  +C \int_{\Reals^3} h^2 \M  \mu^{-1}\;dv\\
		= &   - c\int_{\Reals^3}  \frac{| \nabla h|^2}{\langle v \rangle^3}  \mu^{-1} \;dv \; {{-  \frac{1}{2} \int_{\Reals^3} \partial_t  \mu^{-1} h^2  \;dv }}  + C  \int_{\Reals^3} h^2 \M  \mu^{-1}\;dv,
	\end{align*} 
	which finishes the proof. 
\end{proof}

The remaining estimates of linear near-field terms are contained in the next lemma. 
\begin{lemma}
	The following bounds hold for some constant $C>0$ 
	\begin{align} \label{come} 
		 \int_{\Reals^3} Eg h \mu^{-1} \;dv &\le  C \frac{m}{T^{\frac{9}{4}}}\|h \|_{L^2_{\mu}}\|g \|_{L^2_{\M}}, \\
		\label{Q(g,E)}
		 \int_{\Reals^3} Q(g,E)h  \mu^{-1} \;dv &\leq  C \left( \frac{m}{T^{\frac{9}{4}}}\|h \|_{L^2_{\mu}}\|g \|_{L^2_{\M}} +  {{\frac{m}{T^{\frac{5}{2}}}}}\|h \|_{L^2_{\mu}}\|g \|_{L^2_{\M}}\right).
	\end{align}
\end{lemma}

\begin{proof}
	After substitution, the integral reads 
	\begin{align*} 
		\int_{\Reals^3} Eg h \mu^{-1} \;dv \le  \frac{ m(t)}{T^3} \int_{\Reals^3} gh \;dv \le    \frac{ m(t)}{ T^{3}} \left( \int_{\Reals^3} h^2 \mu^{-1} \;dv\right)^{\frac{1}{2}} \left( \int_{\Reals^3} g^2 \mu \;dv\right)^{\frac{1}{2}} ,
	\end{align*}
	and we observe
	\begin{align*}
	\int_{\Reals^3} g^2 \mu \;dv \le C T^{3/2} \int_{\Reals^3} g^2 \M^{-1} \;dv .
	\end{align*}
	This implies 
	\begin{align*}
		\int_{\Reals^3} Eg h \mu^{-1} \;dv \leq C \frac{ T^{3/4} m}{T^{3}}   \left( \int_{\Reals^3} h^2 \mu^{-1} \;dv\right)^{\frac{1}{2}} \left( \int_{\Reals^3} g^2 \M^{-1}  \;dv\right)^{\frac{1}{2}} ,
	\end{align*}
	and  concludes the proof of~\eqref{come}.

	To show~\eqref{Q(g,E)} let us first write it as $ Q(g,E) = \Tr(A[g]\nabla^2E) + Eg$ using~\eqref{eq:nonDiv}. The second term can be estimated like~\eqref{come}.  For the first one, notice that 
	\begin{align*}
	\nabla^2 E = \left( \frac{1}{T^2}v \otimes v - \frac{1}{T} Id\right)E,
	\end{align*} 
	and therefore 
	\begin{align*}
		\int_{\Reals^3} \Tr(A[g]\nabla^2E) h \mu^{-1} \;dv \le  \sup \langle A[|g|]e,e\rangle  \int_{\Reals^3} Eh \left( \frac{1}{T}+  \frac{|v|^2}{T^2} \right)\mu^{-1} \;dv.
	\end{align*}
	The term $\langle A[g]e,e\rangle$ can be estimated using~\eqref{Ahee} for $g$ and $T=1$. The integral $\int_{\Reals^3} Eh \mu^{-1} \;dv$ is bounded as follows
	\begin{align*} 
		\int_{\Reals^3} Eh \mu^{-1} \;dv &= \delta^\alpha \frac{1}{T^3} \int_{\Reals^3} h \;dv \le \frac{\delta^\alpha }{T^3} \left( \int_{\Reals^3}  h^2 \mu^{-1} \;dv\right)^{1/2}  \underbrace{\left( \int_{\Reals^3}  \mu \;dv\right)^{1/2} }_{ \leq CT^{3/2}}  \\ 
		& \le C\frac{\delta^\alpha }{T^{3/2}} \|h \|_{L^2_{\mu}}, 
	\end{align*}
	while the other term is bounded by
	\begin{align*}
		\int_{\Reals^3} Eh  |v|^2\mu^{-1} \;dv  =& \delta^\alpha \frac{1}{T^3} \int_{\Reals^3} h |v|^2 \;dv  \le \frac{\delta^\alpha }{T^3} \left( \int_{\Reals^3}  h^2 \mu^{-1} \;dv\right)^{1/2}  \underbrace{\left( \int_{\Reals^3}  \mu |v|^4 \;dv\right)^{1/2} }_{ \leq C T^{3/2+1}}  \\
		&  \le  C\frac{\delta^\alpha }{T^{1/2}} \|h \|_{L^2_{\mu}}. 
	\end{align*}
	Combining the bounds gives 
	\begin{align*}
	\int_{\Reals^3} \Tr(A[g]\nabla^2E) h \mu^{-1} \;dv \le  C  \frac{\delta^\alpha}{T^{5/2}} \|h \|_{L^2_{\mu}}\|g \|_{L^2_{\M}}.
	\end{align*}	
\end{proof}

\subsection{Estimates on the far-field scale} 

In this section, we provide the estimates for the linear terms in the far-field. On the far-field scale, we can make use of the coercivity provided by the operator $L_\M$, which is the content of the following lemma.
\begin{lemma}[Far-field coercivity]
	The following coercivity estimate holds, where $D_\M(g)$ is given by~\eqref{eq:DMG}
	\begin{align} \label{eq:coercImproved}
		\int_{\Reals^3} g L_\M(g) \M^{-1} \ud{v} \leq - D_\M(g) + C\|g\|^2_{L^2_\M}.
	\end{align}
\end{lemma}
\begin{proof}
	From Lemma~\ref{lem:firstcoerc}, we directly obtain
	\begin{align*}
	\int_{\Reals^3} g L_\M(g) \M^{-1} \ud{v} \leq - D_\M(P^\perp_\M g).
\end{align*}
Here $P^\perp_\M$ is defined through the projection onto the kernel of $L_\M$. 
Now the parallel component satisfies the estimate
\begin{align*}
	D_\M (P_\M g) \leq C \|g\|^2_{L^2_\M},
\end{align*}
and for $D_\M$ we have
\begin{align*}
	D_\M (g) \leq C \left( D_\M (P_\M g) +  D_\M (P^\perp_\M g) \right) .
\end{align*} 
Combining these estimates we obtain the claim. 
\end{proof}

With the coercivity estimate at hand, we can control the other terms appearing in the equation for $g$.

\begin{lemma}[Linear far-field terms]
	We have the following estimates (with $S_\M$ as defined in~\eqref{def:SM}):
	\begin{align}
		\int_{\Reals^3} g A[E]:\nabla^2 g \M^{-1} \ud{v} &\leq - D_{E}(g) + C\frac{m}{T} \|g\|^2_{L^2_\M} ,  \label{eq:AEg}\\
		\int_{\Reals^3} g S_\M \M^{-1} \ud{v} &\leq C m \|g\|_{L^2_\M} , \label{eq:SM}\\
		\int_{\Reals^3} g A[h]:\nabla^2 \M \ud{v}&\leq C T \|h\|_{L^2_\mu}  \|g\|_{L^2_\M}, \label{eq:AhM}
	\end{align}
	where the dissipation $D_{E}(g)$ is given by
	\begin{align*}
		D_{E}(g) = \int_{\Reals^3}\nabla g A[E] \nabla g \M^{-1}\ud{v} + \frac12 \int_{\Reals^3} g^2 E \M^{-1}\ud{v}.  
	\end{align*}
\end{lemma}
\begin{proof}
	For the first estimate, we integrate by parts
	\begin{align*}
		\int_{\Reals^3} g A[E]:\nabla^2 g \M^{-1} \ud{v} &= - \int_{\Reals^3} \nabla g A[E] \nabla g \M^{-1} +\frac12 \int_{\Reals^3} g^2 \nabla \cdot (\nabla \cdot (A[E] \M^{-1})) \ud{v}\\
		&= - D_{E}(g) + \int_{\Reals^3} g^2 \big(2 \nabla \cdot A[E] \nabla (\M^{-1}) + A[E]:\nabla^2 (\M^{-1}) \big) \ud{v},
	\end{align*}
	where we have used~\eqref{eq:Aa} and~\eqref{eq:aDelta} to find
	\begin{align*}
		\nabla \cdot (\nabla \cdot A[E]) = \nabla \cdot (\nabla a[E]) = -E.
	\end{align*}
	Inserting the explicit form of $E$ we obtain
	\begin{align*}		
		\int_{\Reals^3} g A[E]:\nabla^2 g \M^{-1} \ud{v}&\leq -D_{E}(g) + C \frac{m}{T} \|g\|^2_{L^2_\M}.
	\end{align*}
	The second claim follows by
	\begin{align*}
		\int_{\Reals^3} g S_\M \M^{-1}\ud{v} &= \int_{\Reals^3} g A[E]:\nabla^2 \M \M^{-1}\ud{v} \\
			&\leq  C m \|g\|_{L^2_\M} \frac{1}{\sqrt{T}} \left(\int_{\Reals^3}\frac{1}{ \vT^2 }  \M \ud{v} \right)^\frac12  \\
			&\leq C m \|g\|_{L^2_\M}.
	\end{align*}
	Finally, the last estimate follows by
	\begin{align*}
		\int_{\Reals^3} g A[h]:\nabla^2 \M \ud{v} &\leq C\|g\|_{L^2_\M} \|A[h]\|_{L^\infty}  \\
			&\leq C  T \|g\|_{L^2_\M} \|h\|_{L^2_E},
	\end{align*}
	which finishes the proof of the lemma. 
\end{proof}

\section{Estimate of the nonlinear terms}

\subsection{Preparatory estimates}
We start with the following preparatory estimate. For the control of nonlinear terms, we need fine estimates on the matrix $A[h]$ which are the content of the following lemma.  
\begin{lemma}
	Assume that $T \leq 1$. Then the matrix $A[h]$ satisfies (here $|e|=1$)
	\begin{align}
		\label{Ahvv}
		| \langle A[h]v,v\rangle |  &\le \;  \frac{C}{1+\frac{|v|}{\sqrt{T}}}\;  \|h\|_{L^2_\mu} \; T^{2}, \\
		\label{Ahee}
		 | \langle A[h]e,e\rangle |  &\le  a[|h|] \le  \; \frac{C}{1+\frac{|v|}{\sqrt{T}}}\;  \|h\|_{L^2_\mu} \; T, \\
		 \label{Ahv}
		 |A[h]v| &\le    \frac{C}{1+\frac{|v|}{\sqrt{T}}} \|h\|_{L^2_\mu} T^{\frac{3}{2}}, \\
		 \label{nabla-a-est}
		 | \nabla \cdot ( A[h])|  &\leq C \left( \int_{\Reals^3} \frac{|\nabla h|^2}{\langle w \rangle^3}\mu^{-1} \;dw\right)^{\frac{1}{2}}   \frac{T}{1 + \frac{|v|}{\sqrt{T}}}, \\
		 \label{nabla-a-v}
		  | \nabla \cdot  A[h] \cdot v |   &\leq C \left( \int_{\Reals^3} \frac{|\nabla h|^2}{\langle w \rangle^3}\mu^{-1} \;dw\right)^{\frac{1}{2}}   \frac{T^{\frac{3}{2}}}{1 + \frac{|v|}{\sqrt{T}}}.
	\end{align}
\end{lemma}
\begin{proof} 
{\em{Proof of~\eqref{Ahvv}.}} First note that
\begin{align*}
\langle A[h]v,v\rangle = \int_{\Reals^3} \langle \Pi(v,w)w, w \rangle \frac{h(w)}{|v-w|}\;dw. 
\end{align*} 
The Cauchy-Schwarz inequality implies
\begin{align*}
	| \langle A[h]v,v\rangle  | \le \int_{\Reals^3} \frac{ |w|^2 |h(w)|}{|v-w|} \;dw  \le \left( \int_{\Reals^3} h^2 \mu^{-1}\;dv \right)^{1/2} \left(\int_{\Reals^3} \frac{ |w|^4 \mu(w)}{|v-w|^2} \;dw  \right)^{1/2},
\end{align*}
where $\mu(w) = T^{3/2} e^{-\frac{|w|^2}{2T}}$. Let us compute the last integral: 
\begin{align*}
	\int_{\Reals^3} \frac{ |w|^4 \mu(w)}{|v-w|^2} \;dw  = T^{\frac{3}{2}}  \int_{\Reals^3} \frac{ |w|^4 e^{-\frac{|w|^2}{2T}}}{|v-w|^2} \;dw. 
\end{align*}
Using the change of variable $ y = \frac{w}{\sqrt{T}}$ we get 
\begin{align*}
\int_{\Reals^3} \frac{ |w|^4 e^{-\frac{|w|^2}{2T}}}{|v-w|^2} \;dw = T^{1+3/2} \int_{\Reals^3} { |y|^4 e^{-\frac{|y|^2}{2}}}{\left|\frac{v}{\sqrt{T}}-y\right|^{-2}} \;dy.
\end{align*} 
We now distinguish the cases $|\tilde{v}| \leq 1$ and $|\tilde{v}|\geq 1$ where for simplicity $  \tilde{v}:=\frac{|v|}{\sqrt{T}}$. 
First, assume $\frac{|v|}{\sqrt{T}}\le 1$, which allows us to estimate  
\begin{align*}
	\int_{\Reals^3} \frac{ |y|^4 e^{-\frac{|y|^2}{2}}}{|\tilde{v}-y|^2} \;dy = & \int_{B(\tilde{v},1)} \frac{ |y|^4 e^{-\frac{|y|^2}{2}}}{|\tilde{v}-y|^2} \;dy +  \int_{B^c(\tilde{v},1)} \frac{ |y|^4 e^{-\frac{|y|^2}{2}}}{|\tilde{v}-y|^2} \;dy\\
	\le & C   \int_{B(\tilde{v},1)} \frac{1}{|\tilde{v}-y|^2} \;dy + C \int_{\Reals^3}  |y|^4 e^{-\frac{|y|^2}{2}}\;dy \\
	\leq  & C\left( \; 1+   \int_{\Reals^3} e^{-|y|^2} y^4  \;dy \right) \leq   \; C.
\end{align*}
If instead $|\tilde v| \ge 1$, we can bound:
\begin{align*}
	\int_{\Reals^3} \frac{ |y|^4 e^{-\frac{|y|^2}{2}}}{|\tilde{v}-y|^2} \;dy = & \int_{B(\tilde{v},|\tilde{v}|/2)} \frac{ |y|^4 e^{-\frac{|y|^2}{2}}}{|\tilde{v}-y|^2} \;dy +  \int_{B^c(\tilde{v},|\tilde{v}|/2)} \frac{ |y|^4 e^{-\frac{|y|^2}{2}}}{|\tilde{v}-y|^2} \;dy\\
	\le & C    e^{-\frac{|\tilde{v}|^2}{8}} \int_{B(\tilde{v},|\tilde{v}|/2)}  \frac{1}{|\tilde{v}-y|^2}\;dy +  \frac{C}{|\tilde{v}|^2} \int_{\Reals^3}  |y|^4 e^{-\frac{|y|^2}{2}}\;dy \\
	\leq &  C \left(  e^{-\frac{|\tilde{v}|^2}{8}}  |\tilde{v}| + \frac{ 1}{|\tilde{v}|^2}\right)    	\leq    \frac{ C T}{|v|^2}.
\end{align*}
Combining the estimates from both cases, we find 
\begin{align*}
T^{\frac{3}{2}}  \int_{\Reals^3} \frac{ |w|^4 e^{-\frac{|w|^2}{2T}}}{|v-w|^2} \;dw \le  \frac{C T^{4}}{1 + \frac{|v|^2}{T}},
\end{align*}
which leads to the conclusion 
\begin{align*}
	| \langle A[h]v,v\rangle  | \le  \frac{ CT^{2} }{1+\frac{ |v|}{\sqrt{T}}} \left( \int_{\Reals^3} h^2 \mu^{-1}\;dv \right)^{1/2} .
\end{align*}

{\em{Proof of~\eqref{Ahee} and \eqref{Ahv}.}}  This estimates follows easily by noting that 
\begin{align*}
	| \langle A[h]e,e\rangle  | \le \int_{\Reals^3} \frac{ |h(w)|}{|v-w|} \;dw,
\end{align*}
after which one can adapt the steps of the proof of~\eqref{Ahvv}, and obtain the desired bound. An analogous argument applies to~\eqref{Ahv}. 

{\em{Proof of~\eqref{nabla-a-est}. }} We start by rewriting 
\begin{align*}
	| \nabla \cdot ( A[h])| \le \int_{\Reals^3} \frac{|\nabla h|}{|v-w|}\;dw = \int_{B(0,1)} \frac{|\nabla h|}{|v-w|}\;dw+\int_{B^c(0,1)} \frac{|\nabla h|}{|v-w|}\;dw.
\end{align*} 
For the integral in the ball we use the Cauchy-Schwarz inequality and get 
\begin{align*}
	\int_{B(0,1)} \frac{|\nabla h|}{|v-w|}\;dw &\leq \left( \int_{B(0,1)} \frac{|\nabla h|^2}{\langle w \rangle^3}\mu^{-1} \;dw\right)^{\frac{1}{2}} \left( \int_{B(0,1)}  \frac{1}{|v-w|^2}\mu \;dw \right)^{\frac{1}{2}} \\
	&\leq C \left( \int_{\Reals^3} \frac{|\nabla h|^2}{\langle w \rangle^3}\mu^{-1} \;dw\right)^{\frac{1}{2}}  \frac{T}{1 + \frac{|v|}{\sqrt{T}}} .
\end{align*}
Similarly, in the complement of the ball, 
\begin{align*}
	\int_{B^c(0,1)} \frac{|\nabla h|}{|v-w|}\;dw &\leq \left( \int_{B^c(0,1)} \frac{|\nabla h|^2}{\langle w \rangle^3}\mu^{-1} \;dw\right)^{\frac{1}{2}}   \left( \int_{B^c(0,1)}  \frac{|w|^3}{|v-w|^2}\mu \;dw \right)^{\frac{1}{2}} \\
	&\leq C  \left( \int_{\Reals^3} \frac{|\nabla h|^2}{\langle w \rangle^3}\mu^{-1} \;dw\right)^{\frac{1}{2}}   \frac{T^{\frac{7}{4}}}{1 + \frac{|v|}{\sqrt{T}}} .
\end{align*}
Since $T \leq  1$, the leading order term is as claimed:  
\begin{align*}
	| \nabla \cdot ( A[h])|  \leq C \left( \int_{\Reals^3} \frac{|\nabla h|^2}{\langle w \rangle^3}\mu^{-1} \;dw\right)^{\frac{1}{2}}   \frac{T}{1 + \frac{|v|}{\sqrt{T}}}. 
\end{align*}

{\em{Proof of~\eqref{nabla-a-v}.}}  We again separate the contributions of $w$ small and large  
\begin{align*}
	| v \nabla \cdot ( A[h])|  \le \int_{\Reals^3} \frac{|w||\nabla h|}{|v-w|}\;dw = \int_{B(0,1)} \frac{|w||\nabla h|}{|v-w|}\;dw+\int_{B^c(0,1)} \frac{|w||\nabla h|}{|v-w|}\;dw.
\end{align*} 
For the integral in the ball we use the Cauchy-Schwarz inequality and get 
\begin{align*}
	\int_{B(0,1)} \frac{|w||\nabla h|}{|v-w|}\;dw \le \left( \int_{B(0,1)} \frac{|\nabla h|^2}{\langle w \rangle^3}\mu^{-1} \;dw\right)^{\frac{1}{2}} \left( \int_{B(0,1)}  \frac{|w|^2}{|v-w|^2}\mu \;dw \right)^{\frac{1}{2}} \\
	\le C \left( \int_{\Reals^3} \frac{|\nabla h|^2}{\langle w \rangle^3}\mu^{-1} \;dw\right)^{\frac{1}{2}}  \frac{T^{\frac{3}{2}}}{1 + \frac{|v|}{\sqrt{T}}} .
\end{align*}
Similarly, in the complement of the ball, 
\begin{align*}
	\int_{B^c(0,1)} \frac{|w||\nabla h|}{|v-w|}\;dw \le C \left( \int_{B^c(0,1)} \frac{|\nabla h|^2}{\langle w \rangle^3}\mu^{-1} \;dw\right)^{\frac{1}{2}}   \left( \int_{B^c(0,1)}  \frac{|w|^5}{|v-w|^2}\mu \;dw \right)^{\frac{1}{2}} \\
	\le  C \left( \int_{\Reals^3} \frac{|\nabla h|^2}{\langle w \rangle^3}\mu^{-1} \;dw\right)^{\frac{1}{2}}   \frac{T^{\frac{9}{4}}}{1 + \frac{|v|}{\sqrt{T}}} .
\end{align*}
Since the leading order term is the one with $T^{\frac{3}{2}}$, we have that 
\begin{align*}
	|v \nabla \cdot ( A[h])|  \leq C \left( \int_{\Reals^3} \frac{|\nabla h|^2}{\langle w \rangle^3}\mu^{-1} \;dw\right)^{\frac{1}{2}}   \frac{T^{\frac{3}{2}}}{1 + \frac{|v|}{\sqrt{T}}}. 
\end{align*} 

\end{proof}

\begin{lemma}
	For $h\in L^2_\mu$, there exists $\psi \in L^6(\Reals^3)$ such that the following estimates hold
	\begin{equation} \label{est:nablaA} 
	\begin{aligned}
		|A[\nabla h](v)| &\leq C T^{\frac34} \left( e^{-|v|} \psi(v) +  \|h\|_{L^2_\mu}\langle v \rangle ^{-2} \right)  , \\
		\|\psi\|_{L^6(\Reals^3)} &\leq \|h\|_{L^2_\mu}.
	\end{aligned}
	\end{equation} 
\end{lemma}

\begin{proof}
	We first rewrite
	\begin{align*}
		A[\nabla h](v)  &= \int_{\Reals^3} \frac{\Pi(v-w)}{|v-w|} \nabla h(w)  \ud{w} \\
		&= -\int_{\Reals^3} \nabla \cdot\big(\frac{\Pi(v-w)}{|v-w|}\big)  h(w)  \ud{w} \\
		&= \int_{\Reals^3}   \frac{h(w)  (v-w)}{|v-w|^3} \ud{w} .
	\end{align*}
	To estimate this, we split the integral as:
	\begin{align*}
		A[\nabla h](v)  	&= \int_{B_1(v)}   \frac{h(w)  (v-w)}{|v-w|^3} \ud{w} + \int_{B_1^c(v)}   \frac{h(w) (v-w)}{|v-w|^3} \ud{w} \\
				&=J_{1} + J_{2}	.
	\end{align*}
	For the first term we use
	\begin{align*}
		|J_{1}| &\leq C T^{\frac34}e^{-|v|} \int_{B_1(v)}   \frac{|h(w)| \mu^{-\frac12} }{|v-w|^2} \ud{w},
	\end{align*}
	and hence there exists a non-negative $\psi \in L^6$ such that 
	\begin{align*}
		|J_{1}|		&\leq CT^{\frac34} \psi(v), \\
		\|\psi\|_{L^6}	&\leq \|h\|_{L^2_\mu }.
	\end{align*}
	For the remaining term $J_{2}$, we use the straightforward estimate 
	\begin{align*}
		|J_{2}| &\leq C\int_{B_1^c(v)}   \frac{h(w)  }{\langle v-w \rangle ^2} \ud{w} \\
		 &\leq C \|h\|_{L^2_\mu} \left(\int_{B_1(v)^c} \frac{\mu(w) }{\langle v-w\rangle^4 } \right)^\frac12 \\
		 &\leq C  \frac{\|h\|_{L^2_\mu} T^\frac34}{\langle v\rangle^2 } .
	\end{align*}
	Collecting the estimates for $J_1$, and $J_2$ yields the claim. 
\end{proof}

\subsection{Estimates for the nonlinear far-field terms}

\begin{lemma}
	For any $\eps>0$, there exists $C=C(\eps)>0$ such that:
	\begin{align}
		\left|\int  g^3 \M^{-1} \ud{v}\right| &\leq  C \|g\|_{L^2_\M}^6 + \eps  D_\M(g) , \label{eq:gnonlin1} \\		
		\left|\int A[h] : \nabla^2 g g \M^{-1} \ud{v}\right| &\leq C T^\frac34 \|h\|_{L^2_\mu}\left( \|g\|^2_{L^2_\M}+ D_\M(g) \right)+ C T^3 \|g \|_{L^2_\M}^2 \|h\|_{L^2_\mu}^4+ \eps   D_\M(g), \label{eq:gnonlin2} \\
			\left|\int_{\Reals^3} Q(g,g) g \M^{-1} \ud{v}\right| &\leq C  \|g\|_{L^2_\M}\left( \|g\|^2_{L^2_\M}+ D_\M(g) \right)+C \|g\|_{L^2_\M}^6 + \eps  D_\M(g). \label{eq:gnonlin3} 
	\end{align}
\end{lemma}
\begin{proof}
	We start by proving~\eqref{eq:gnonlin1}. From the Cauchy-Schwarz inequality we infer
	\begin{align*}
		\left| \int  g^3 \M^{-1} \ud{v}\right| &\leq \|g\|_{L^2_\M} \left( \int (g^2 \M^{-\frac12})^2 \ud{v}\right)^\frac12  \\
		&\leq  \|g\|_{L^2_\M} \|g^2 \M^{-\frac12}\|_{L^1}^\frac14 \|g^2 \M^{-\frac12}\|_{L^3}^\frac34 \\
		&\leq C \|g\|_{L^2_\M}^\frac32  \|g^2 \M^{-\frac12}\|_{L^3}^\frac34 \\
			&\leq C \|g\|_{L^2_\M}^\frac32  D_\M(g)^\frac34 \\
			&\leq C \|g\|_{L^2_\M}^6 + \eps  D_\M(g).
	\end{align*}
	To prove~\eqref{eq:gnonlin2}, we first integrate by parts and obtain 
	\begin{align*}
		\int_{\Reals^3} A[h] : \nabla^2 g g \M^{-1} \ud{v} = - &\int_{\Reals^3}   \nabla g  A[h] \nabla g \M^{-1} \ud{v} - \int_{\Reals^3} v A[h]  g  \nabla g   \M^{-1} \ud{v} \\
		- &\int_{\Reals^3} \nabla \cdot  A[h]  g  \nabla g   \M^{-1} \ud{v} = -(I_1+I_2+I_3).  
	\end{align*}
	For the estimate of $I_1$, we use the estimate for $A[h]$ ($\hat{v}= \frac{v}{|v|}$):
	\begin{align}
		|\hat v A[h] \hat v| \leq \frac{T\|h\|_{L^2_\mu} }{\langle v/\sqrt{T}\rangle^3 } .
	\end{align}
	Since $\nabla g(v) \parallel \hat v$, we have 
	\begin{align*}
		|I_1| &\leq  \left| \int_{\Reals^3}   \nabla g  A[h] \nabla g \M^{-1} \ud{v} \right|  \\
		&\leq  CT \int_{\Reals^3}   \frac{|\nabla g|^2}{\langle v/\sqrt T\rangle^3  }  \M^{-1} \ud{v}  \|h\|_{L^2_\mu} \\
		&\leq C T D_\M(g)  \|h\|_{L^2_\mu}. 
	\end{align*}
	The same argument applies to $I_2$, and we obtain
	\begin{align*}
		|I_2| &\leq \left|  \int_{\Reals^3} v A[h]  g  \nabla g   \M^{-1} \ud{v}\right| \\
			&\leq   CT  \|h\|_{L^2_\mu}\int_{\Reals^3} \frac{|v|   |g|  |\nabla g|}{\langle v/\sqrt T\rangle^3}   \M^{-1} \ud{v} \\
			&\leq   CT  \|h\|_{L^2_\mu}\int_{\Reals^3} \frac{|v|^2   |g|^2  +|\nabla g|^2}{\langle v/\sqrt T\rangle^3}   \M^{-1} \ud{v} \\
			&\leq C T \|h\|_{L^2_\mu} \left( \|g\|^2_{L^2_\M}+ D_\M(g) \right).
	\end{align*}
	The estimate for $I_3$ follows from~\eqref{est:nablaA}. We obtain
	\begin{align*}
		|I_3|&=  \left|\int_{\Reals^3} \nabla \cdot  A[h]  g  \nabla g   \M^{-1} \ud{v}\right| \\
		&\leq   \left|\int_{\Reals^3} [C T^{\frac34} \left( e^{-|v|} \psi(v) +  \|h\|_{L^2_\mu}\langle v \rangle ^{-2} \right) ]  g  \nabla g   \M^{-1} \ud{v}\right|  \\
		&\leq  C T^\frac34 \|h\|_{L^2_\mu}\left( \|g\|^2_{L^2_\M}+ D_\M(g) \right)+C T^{\frac34} \left|\int_{\Reals^3}  e^{-|v|} \psi(v)   g  \nabla g   \M^{-1} \ud{v}\right|\\
			&\leq  C T^\frac34 \|h\|_{L^2_\mu}\left( \|g\|^2_{L^2_\M}+ D_\M(g) \right)+ C T^\frac34 \|h\|_{L^2_\mu}\|g \M^{-\frac12} e^{-\frac12|v|}\|_{L^3} D_\M(g)^\frac12 \\
			&\leq  C T^\frac34 \|h\|_{L^2_\mu}\left( \|g\|^2_{L^2_\M}+ D_\M(g) \right)+ C T^\frac34 \|g \|_{L^2_\M}^\frac12 \|h\|_{L^2_\mu}  D_\M(g)^\frac34\\
			&\leq  C T^\frac34 \|h\|_{L^2_\mu}\left( \|g\|^2_{L^2_\M}+ D_\M(g) \right)+ C T^3 \|g \|_{L^2_\M}^2 \|h\|_{L^2_\mu}^4+ \eps   D_\M(g).
	\end{align*}
Finally, for the proof of~\eqref{eq:gnonlin3}, we use the non-divergence form of the equation to rewrite
\begin{align*}
		\left|\int_{\Reals^3} Q(g,g) g \M^{-1} \ud{v}\right| &= \left|\int_{\Reals^3} (A[g]: \nabla^2g g+ g^3) \M^{-1} \ud{v}\right|.
\end{align*}
We then use~\eqref{eq:gnonlin1} as well as~\eqref{eq:gnonlin2} with $h=g$ and $T=1$ to infer:
\begin{align*}
	\left|\int_{\Reals^3} Q(g,g) g \M^{-1} \ud{v}\right| &\leq C  \|g\|_{L^2_\M}\left( \|g\|^2_{L^2_\M}+ D_\M(g) \right)+C \|g\|_{L^2_\M}^6 + \eps  D_\M(g).
\end{align*}

\end{proof}

\subsection{Estimates for the nonlinear near-field terms} 

We turn to the estimates of the nonlinear near-field contributions. To this end, we recall that  $h$ satisfies the equation introduced in~\eqref{eq:hgcoupled}.  

\begin{lemma}
The following inequality holds for any $\eps>0$ and $C=C(\eps)>0$ large enough
\begin{align}\label{gh}
{ \int h^2 g  \mu^{-1} \;dv \le  C \| g\|^4_{L^2_{\M}} \|h\|^2_{L^2_\mu} + C \frac{1}{T^{3/2}} \| g\|_{L^2_{\M}}  \|h\|^2_{L^2_\mu} +  \varepsilon  \int \frac{1}{ \langle v \rangle^3} \mu^{-1} |\nabla h|^2 \;dv.} 
\end{align}
\end{lemma}

\begin{proof} 
We use the $L^p$-interpolation 
\begin{align*}
\|f \|_{L^2} \le \|f \|^{1/4}_{L^1}\|f \|^{3/4}_{L^3},
\end{align*} 
and the Cauchy-Schwarz inequality to obtain  
\begin{align*}
\int_{\Reals^3} h^2 g  \mu^{-1} \;dv  \le \left( \int_{\Reals^3} g^2 \M^{-1} dv \right)^{1/2} \left( \int_{\Reals^3} ( h^2 \M^{1/2}\mu^{-1})^2\; dv \right)^{1/2}  \\
\le  \left( \int_{\Reals^3} g^2 \M^{-1} dv \right)^{1/2} \left( \int_{\Reals^3} h^2 \M^{1/2}\mu^{-1}\; dv \right)^{1/4} \left( \int_{\Reals^3} h^6 \M^{3/2}\mu^{-3}\; dv \right)^{1/4}.
\end{align*}
Applying the Sobolev embedding yields 
\begin{align*}
 \left( \int_{\Reals^3} h^6 \M^{3/2}\mu^{-3}\; dv \right)^{1/4} \le \left( \int_{\Reals^3} | \nabla( h \M^{1/4} \mu^{-1/2})|^2\;dv\right)^{3/4},
\end{align*}
which we can use to bound 
\begin{align*}
\int_{\Reals^3} h^2 g  \mu^{-1} \;dv  \le & \left( \int_{\Reals^3} g^2 \M^{-1} dv \right)^{1/2} \left( \int_{\Reals^3} h^2 \M^{1/2}\mu^{-1}\; dv \right)^{1/4}  \\
&\cdot  \left( \int_{\Reals^3} \M^{1/2}\mu^{-1} |\nabla h|^2 \;dv + \int_{\Reals^3} h^2 |\nabla(\M^{1/4} \mu^{-1/2})|^2\right)^{3/4} \\
\le& \left( \int_{\Reals^3} g^2 \M^{-1} dv \right)^{1/2} \left( \int_{\Reals^3} h^2 \M^{1/2}\mu^{-1}\; dv \right)^{1/4}  \\
&\cdot  \left( \int_{\Reals^3} \frac{1}{ \langle v \rangle^3} \mu^{-1} |\nabla h|^2 \;dv + \int_{\Reals^3} h^2 |\nabla(\M^{1/4} \mu^{-1/2})|^2\right)^{3/4}  \\
 \le &\left( \int_{\Reals^3} g^2 \M^{-1} dv \right)^{1/2} \left( \int_{\Reals^3} h^2 \mu^{-1}\; dv \right)^{1/4}   \left( \int_{\Reals^3} \frac{1}{ \langle v \rangle^3} \mu^{-1} |\nabla h|^2 \;dv \right)^{3/4} \\
&+ \left( \int_{\Reals^3} g^2 \M^{-1} dv \right)^{1/2} \left( \int_{\Reals^3} h^2 \mu^{-1}\; dv \right)^{1/4}    \left( \int_{\Reals^3} h^2 |\nabla(\M^{1/4} \mu^{-1/2})|^2\right)^{3/4}.
\end{align*}
We now compute the gradient of $\M^{1/4} \mu^{-1/2}$ to find 
\begin{align*}
 |\nabla(\M^{1/4} \mu^{-1/2})|^2 \le & \; \mu^{-1} | \nabla \M^{1/4}|^2 + \M^{1/2} | \nabla  \mu^{-1/2}|^2 \\
 \le & \;  \mu^{-1} +  \M^{1/2} \frac{|v|^2}{T^2}\mu^{-1}  \\
 \le & \;   \mu^{-1} + \frac{1}{T^2}\mu^{-1} .
\end{align*}
Using the above expression we have 
\begin{align*}
\int_{\Reals^3} h^2 g  \mu^{-1} \;dv  \le &  \left( \int_{\Reals^3} g^2 \M^{-1} dv \right)^{1/2} \left( \int_{\Reals^3} h^2 \mu^{-1}\; dv \right)^{1/4}   \left( \int_{\Reals^3} \frac{1}{ \langle v \rangle^3} \mu^{-1} |\nabla h|^2 \;dv \right)^{3/4} \\
&+  \left( 1 + \frac{1}{T^{\frac{3}{2}}}\right) \left( \int_{\Reals^3} g^2 \M^{-1} dv \right)^{1/2}  \int_{\Reals^3} h^2 \mu^{-1}\; dv  \\
\le & C\left( \int_{\Reals^3} g^2 \M^{-1} dv \right)^{2}  \int_{\Reals^3} h^2 \mu^{-1}\; dv  +   \varepsilon  \int_{\Reals^3} \frac{1}{ \langle v \rangle^3} \mu^{-1} |\nabla h|^2 \;dv  \\
&+  C\left( 1 + \frac{1}{T^{\frac{3}{2}}}\right) \left( \int_{\Reals^3} g^2 \M^{-1} dv \right)^{1/2}  \int_{\Reals^3} h^2 \mu^{-1}\; dv .
\end{align*}

\end{proof}

In the next lemma we show how to bound, in weighted $L^2$-norm, the  drift term $h^2$:

\begin{lemma}
For all $\eps>0$ there exists $C=C(\eps)>0$ such that  
\begin{align} \label{h3}
{ \int_{\Reals^3} h^3 \mu^{-1} \;dv \le C  \|h\|^3_{L^2_\mu} + C T^{\frac{15}{2}}  \|h\|^6_{L^2_\mu}  +  \varepsilon \int_{\Reals^3} \frac{1}{ \langle v \rangle^3} \mu^{-1} |\nabla h|^2 \;dv . }
\end{align}
\end{lemma}

\begin{proof}
Similar to the previous estimate, we can bound the weighted $L^3$ norm as 
\begin{align*}
\int_{\Reals^3} h^3 \mu^{-1} \;dv \le & \left( \int_{\Reals^3} h^2  \mu^{-1} \;dv\right)^{1/2} \left( \int_{\Reals^3} (h^2  \mu^{-1/2})^2 \;dv\right)^{1/2} \\
\le & \left( \int_{\Reals^3} h^2  \mu^{-1} \;dv\right)^{1/2} \left( \int_{\Reals^3} h^2  \mu^{-1/2} \;dv\right)^{1/4}  \left( \int_{\Reals^3} | \nabla ( h \mu^{-1/4})|^2\;dv\right)^{3/4} \\
\le & \left( \int_{\Reals^3} h^2  \mu^{-1} \;dv\right)^{1/2} \left( \int_{\Reals^3} h^2  \mu^{-1/2} \;dv\right)^{1/4}   \left( \int_{\Reals^3}  h^2| \nabla \mu^{-1/4}|^2\;dv\right)^{3/4}  \\
& +  \left( \int_{\Reals^3} h^2  \mu^{-1} \;dv\right)^{1/2} \left( \int_{\Reals^3} h^2  \mu^{-1/2} \;dv\right)^{1/4}   \left( \int_{\Reals^3}  \mu^{-1/2}| \nabla h|^2\;dv\right)^{3/4}  \\
\le & \left( \int_{\Reals^3} h^2  \mu^{-1} \;dv\right)^{1/2} \left( \int_{\Reals^3} h^2  \mu^{-1/2} \;dv\right)^{1/4}   \left( \frac{1}{16 T^2} \int_{\Reals^3}  h^2 |v|^2\mu^{-1/2} \;dv\right)^{3/4}  \\
& +  \left( \int_{\Reals^3} h^2  \mu^{-1} \;dv\right)^{1/2} \left( \int_{\Reals^3} h^2  \mu^{-1/2} \;dv\right)^{1/4}   \left( \int_{\Reals^3}  \mu^{-1/2}| \nabla h|^2\;dv\right)^{3/4}  .
\end{align*}
Since  $\mu = T^{3/2}e^{-\frac{|v|^2}{2T}} $ we have
$$
\mu^{-1/2}  =  \mu^{1/2}  \mu^{-1}  = {T^{3/4}} e^{-\frac{|v|^2}{4T}}  \mu^{-1} 
$$
and
$$
|v|^2  \mu^{-1/2}    =  |v|^2 \;T^{3/4}\; e^{\frac{-|v|^2}{4T}} \mu^{-1} = T^{7/4} \; \frac{ |v|^2}{T} e^{\frac{-|v|^2}{4T}} \mu^{-1} \le  T^{7/4} \mu^{-1},
$$

which implies 
\begin{align*}
\int_{\Reals^3} h^3 \mu^{-1} \;dv &\le  \left( \int_{\Reals^3} h^2  \mu^{-1} \;dv\right)^{1/2} \left(  {T^{3/4}} \int_{\Reals^3} h^2  \mu^{-1} \;dv\right)^{1/4}   \left( \frac{1}{16 T^{1/4}} \int_{\Reals^3}  h^2 \mu^{-1} \;dv\right)^{3/4} \\
 +  &\left( \int_{\Reals^3} h^2  \mu^{-1} \;dv\right)^{1/2} \left(  {T^{3/4}} \int_{\Reals^3} h^2  \mu^{-1} \;dv\right)^{1/4}  \left(  {T^{3/4}} \int_{\Reals^3}   \underbrace{e^{-\frac{|v|^2}{4T}}  \langle v \rangle^3}_{ \le C T^{3/2}}  \frac{\mu^{-1}}{ \langle v \rangle^{3}} | \nabla h|^2\;dv\right)^{3/4}   \\
&\le   \left( \int_{\Reals^3} h^2  \mu^{-1} \;dv\right)^{1+1/2} + T^{30/16}  \left( \int_{\Reals^3} h^2  \mu^{-1} \;dv\right)^{1/2 + 1/4} \left(   \int_{\Reals^3}   \frac{\mu^{-1}}{ \langle v \rangle^{3}} | \nabla h|^2\;dv\right)^{3/4} \\
&\le   \left( \int_{\Reals^3} h^2  \mu^{-1} \;dv\right)^{3/2} + C T^{15/2} \left( \int_{\Reals^3} h^2  \mu^{-1} \;dv\right)^{3} +  \varepsilon \int_{\Reals^3}   \frac{\mu^{-1}}{ \langle v \rangle^{3}} | \nabla h|^2\;dv.
\end{align*}

\end{proof}


The previous estimate will be used to bound the bilinear operator $Q(h,h)$:

\begin{lemma}
The contribution of the quadratic Landau-Coulomb operator to $h$ can be estimated by
\begin{align}\label{Qhh}
{\int_{\Reals^3} Q(h,h) h \mu^{-1} \;dv \leq ( \varepsilon  +  C T   \|h\|_{L^2_{\mu}} ) \int_{\Reals^3}  \frac{| \nabla h|^2}{\langle v \rangle ^3} \mu^{-1} \;dv  +C T  \|h\|^4_{L^2_{\mu}}
 + C\|h\|^3_{L^2_\mu}  }.
\end{align}
\end{lemma}

\begin{proof}
Integration by parts yields 
\begin{align*}
\int_{\Reals^3} Q(h,h) h \mu^{-1} \;dv =  & - \int_{\Reals^3} \langle A[h] \nabla h, \nabla h \rangle \mu^{-1} \;dv - \frac{1}{T}\int_{\Reals^3}  \langle A[h] v, \nabla h \rangle  h \mu^{-1} \;dv \\
& \; - \int_{\Reals^3}  \langle \nabla \cdot (A[h]) ,\nabla h \rangle   \; h  \mu^{-1} \;dv +  \int_{\Reals^3}  h^3  \mu^{-1} \;dv =: I_1 + I_2+I_3+I_4.
 \end{align*}
We start with $I_1$ and split the integral in two parts: 
\begin{align*}
 - \int_{\Reals^3} \langle A[h] \nabla h, \nabla h \rangle \mu^{-1} \;dv& =  - \int_{B(0,1)} \langle A[h] \nabla h, \nabla h \rangle \mu^{-1} \;dv - \int_{B^c(0,1)}  \langle A[h] \nabla h, \nabla h \rangle \mu^{-1} \;dv \\
 &= I_{1,1} + I_{1,2}.
\end{align*}
The term $I_{1,1}$ can be easily estimated using~\eqref{Ahee}: 
\begin{align*}
I_{1,1}  \le \| A[h]\|_{L^\infty}  \int_{B(0,1)} | \nabla h|^2  \mu^{-1} \;dv  \leq C T   \|h\|_{L^2_{\mu}} \int_{\Reals^3}  \frac{| \nabla h|^2}{\langle v \rangle^3}  \mu^{-1} \;dv.
\end{align*}
For $I_{1,2}$ we use that $h$ is radially symmetric and $\nabla h = \frac{v}{|v|} | \nabla h|$. This yields 
\begin{align*}
\int_{B^c(0,1)} | \langle A[h] \nabla h, \nabla h \rangle | \mu^{-1} \;dv \le & \int_{B^c(0,1)}  | \langle A[h] v, v \rangle | \frac{| \nabla h|^2}{|v|^2} \mu^{-1} \;dv \\
\le & \sup_{B^c(0,1)} {|\langle A[h] v, v \rangle||v|}  \int_{\Reals^3}  \frac{| \nabla h|^2}{\langle v \rangle ^3} \mu^{-1} \;dv  \\
\le & \frac{|v|}{1+\frac{|v|}{\sqrt{T}}}\;  \|h\|_{L^2_\mu} \; T^{2}  \int_{\Reals^3}  \frac{| \nabla h|^2}{\langle v \rangle ^3} \mu^{-1} \;dv \\
\le &  \|h\|_{L^2_\mu} \; T^{2+ \frac{1}{2}}  \int_{\Reals^3}  \frac{| \nabla h|^2}{\langle v \rangle ^3} \mu^{-1} \;dv.
\end{align*}
Summarizing, we obtain  
\begin{align*}
|I_1|  \leq C T   \|h\|_{L^2_{\mu}} \int_{\Reals^3}  \frac{| \nabla h|^2}{\langle v \rangle^3}  \mu^{-1} \;dv.
\end{align*}
We split the integral between a ball and its complement also in $I_3$ and get  
\begin{align*}
 \int_{\Reals^3}  \langle \nabla \cdot (A[h]) ,\nabla h \rangle   \; h  \mu^{-1} \;dv =  & \int_{B(0,1)}  \langle \nabla \cdot (A[h]) ,\nabla h \rangle   \; h  \mu^{-1} \;dv \\
 & + \int_{B^c(0,1)}  \langle \nabla \cdot (A[h]) ,\nabla h \rangle   \; h  \mu^{-1} \;dv =:I_{3,1} + I_{3,2}. 
\end{align*}
The first term can be estimated by 
\begin{align*}
|I_{3,1}|  \le& \;  \varepsilon  \int_{\Reals^3}  \frac{| \nabla h|^2}{\langle v \rangle ^3} \mu^{-1} \;dv +  C \int_{B(0,1)}  |\nabla \cdot (A[h]) |^2 h^2 \mu^{-1} \;dv  \\
\le & \left(  \varepsilon + C T^2 \|h\|^2_{L^2_{\mu}}\right)  \int_{\Reals^3}  \frac{| \nabla h|^2}{\langle v \rangle^3}  \mu^{-1} \;dv.
\end{align*}

To estimate  $I_{3,2}$ one takes advantage of the radial symmetry of $h$ and obtains 
\begin{align*}
 \int _{B^c(0,1)}  \langle \nabla \cdot (A[h]) ,\nabla h \rangle   \; h  \mu^{-1} \;dv = &  \int _{B^c(0,1)}  \langle \nabla \cdot (A[h]), v  \rangle \frac{ |\nabla h|}{|v|}  \; h  \mu^{-1} \;dv  \\
 \le &  \varepsilon \int _{B^c(0,1)}   \frac{ |\nabla h|^2}{|v|^3} \mu^{-1} \;dv +   C\int _{B^c(0,1)}  | \langle \nabla \cdot (A[h]), v  \rangle |^2 |v|  \; h^2  \mu^{-1} \;dv \\
\le & \left(  \varepsilon + CT^{\frac{1}{2}+3} \|h\|^2_{L^2_{\mu}}\right)  \int_{\Reals^3}  \frac{| \nabla h|^2}{\langle v \rangle^3}  \mu^{-1} \;dv,
\end{align*}
using~\eqref{nabla-a-v} in the last integral. Summarizing, we have obtained
\begin{align*}
|I_3| \le  \left(  \varepsilon + C T^2 \|h\|^2_{L^2_{\mu}}\right)  \int_{\Reals^3}  \frac{| \nabla h|^2}{\langle v \rangle^3}  \mu^{-1} \;dv.
\end{align*}

For $I_2$ we use the same method as for $I_3$, using estimate for $A[hv]$ instead of $\nabla \cdot (A[h]) $ and taking into account the factor $\frac{1}{T}$ in front of the integral. We write 
\begin{align*}
I_2 = - \frac{1}{T}\int_{B(0,1)}  \langle A[h] v, \nabla h \rangle  h \mu^{-1} \;dv- \frac{1}{T}\int_{B^c(0,1)}  \langle A[h] v, \nabla h \rangle  h \mu^{-1} \;dv =: I_{2,1} + I_{2,2}.
\end{align*}
Similarly we above, we bound $I_{2,1}$ using~\eqref{Ahv} and get  
\begin{align*}
|I_{2,1}| \le  \varepsilon  \int_{\Reals^3}  \frac{| \nabla h|^2}{\langle v \rangle ^3} \mu^{-1} \;dv  +C \frac{1}{T^2} \int_{B(0,1)}  | A[h] v|^2 h^2  \mu^{-1} \;dv \\
\le  \varepsilon  \int_{\Reals^3}  \frac{| \nabla h|^2}{\langle v \rangle ^3} \mu^{-1} \;dv  + C T  \|h\|^4_{L^2_{\mu}}, 
\end{align*}
while we estimate $I_{2,2}$ using the fact that $\nabla h = \frac{v}{|v|} |\nabla h|$, which yields
\begin{align*}
I_{2,2} \le & \; \varepsilon  \int_{\Reals^3}  \frac{| \nabla h|^2}{\langle v \rangle ^3} \mu^{-1} \;dv   +  \frac{C}{T^2} \int_{B^c(0,1)}  | A[h] v|^2 |v| h^2  \mu^{-1} \;dv \\
\le & \; \varepsilon  \int_{\Reals^3}  \frac{| \nabla h|^2}{\langle v \rangle ^3} \mu^{-1} \;dv  +C T^{\frac{3}{2}} \|h\|^4_{L^2_{\mu}}.
\end{align*}
Collecting the estimates we obtain the following bound for $I_2$
$$
|I_2| \leq  \varepsilon  \int_{\Reals^3}  \frac{| \nabla h|^2}{\langle v \rangle ^3} \mu^{-1} \;dv  + C T  \|h\|^4_{L^2_{\mu}}.
$$
We finish the proof combining the above estimates with the ones for  $I_4$~\eqref{h3} and selecting the leading order terms. 
\end{proof}

A similar method can be applied to bound $Q(g,h)$:
\begin{lemma}
For any $\eps>0$ there exists a $C(\eps)>0$ such that 
\begin{align} \label{Q(g,h)}
\int_{\Reals^3} Q(g,h) h \mu^{-1} \;dv   \leq  & \;  ( \varepsilon  + C\|g\|_{L^2_{\M}})  \int_{\Reals^3}  \frac{| \nabla h|^2}{\langle v \rangle^3}  \mu^{-1} \;dv +  C\frac{1}{T^2}  \|h\|^2_{L^2_{\mu}}\|g\|^2_{L^2_{\M}} \nonumber \\
& \; + C \| g\|^4_{L^2_{\M}} \|h\|^2_{L^2_\mu} + C\frac{1}{T^{3/2}} \| g\|_{L^2_{\M}}  \|h\|^2_{L^2_\mu}.
\end{align}
\end{lemma}

\begin{proof}
Integration by parts allows us to rewrite
 \begin{align*}
\int_{\Reals^3} Q(g,h) h \mu^{-1} \;dv =  & - \int_{\Reals^3} \langle A[g] \nabla h, \nabla h \rangle \mu^{-1} \;dv - \frac{1}{T}\int_{\Reals^3}  \langle A[g] v, \nabla h \rangle  h \mu^{-1} \;dv \\
&+  \int_{\Reals^3}  \langle \nabla \cdot (A[g]) ,\nabla (\mu^{-1}h) \rangle   \; h   \;dv.
 \end{align*}
 We integrate by parts again to obtain
 \begin{align*}
 \int_{\Reals^3}  \langle \nabla \cdot (A[g]) ,&\nabla (\mu^{-1}h) \rangle   \; h   \;dv =   \int_{\Reals^3} g h^2 \mu^{-1} \;dv  -  \frac{1}{2}\int_{\Reals^3}  \langle \nabla \cdot (A[g]) ,\nabla h^2\rangle   \;  \mu^{-1}    \;dv \\
 =& \frac{1}{2}  \int_{\Reals^3} g h^2 \mu^{-1} \;dv +  \frac{1}{2}\int_{\Reals^3}  h^2 \nabla \cdot (A[g] )\nabla \mu^{-1}    \;dv \\ 
 =& \frac{1}{2}  \int_{\Reals^3} g h^2 \mu^{-1} \;dv - \int_{\Reals^3} h \langle A[g] \nabla \mu^{-1} , \nabla h  \rangle  \;dv - \frac{1}{2} \int_{\Reals^3}  h^2 \Tr(A[g] \nabla^2 (\mu^{-1})) \;dv\\
  = &\frac{1}{2}  \int_{\Reals^3} g h^2 \mu^{-1} \;dv - \frac{1}{T}\int_{\Reals^3} h  \langle A[gv] , \nabla h \rangle \mu^{-1}   \;dv - \frac{1}{2} \int_{\Reals^3}  h^2 \Tr(A[g] \nabla^2 (\mu^{-1}))  \;dv.
 \end{align*} 
 Collecting all terms, we have shown that  
 \begin{align*}
\int_{\Reals^3} Q(g,h) h \mu^{-1} \;dv =  & - \int_{\Reals^3} \langle A[g] \nabla h, \nabla h \rangle \mu^{-1} \;dv - \frac{2}{T}\int_{\Reals^3}  \langle A[g] v, \nabla h \rangle  h \mu^{-1} \;dv \\
 & - \frac{1}{2} \int_{\Reals^3}  h^2 \Tr(A[g] \nabla^2 (\mu^{-1}))  \;dv+\frac{1}{2}  \int_{\Reals^3} g h^2 \mu^{-1} \;dv
=: J_1 + J_2+J_3.
\end{align*}
Notice that the estimates~\eqref{Ahvv}-\eqref{nabla-a-v} hold also if we replace $h$ with $g$ and set $T=1$. In light of that, the term $J_1$ can be estimated analogous to $I_1$: 
\begin{align*}
|J_1|  \leq C   \|g\|_{L^2_{\M}} \int_{\Reals^3}  \frac{| \nabla h|^2}{\langle v \rangle^3}  \mu^{-1} \;dv.
\end{align*}
The term $J_4$ has been estimated in \eqref{h3}. The term $J_2$ is bounded by 
\begin{align*}
|J_2| \leq  \varepsilon  \int_{\Reals^3}  \frac{| \nabla h|^2}{\langle v \rangle ^3} \mu^{-1} \;dv  + C\frac{1}{T^2}  \|h\|^2_{L^2_{\mu}}\|g\|^2_{L^2_{\mu}},
\end{align*}
using \eqref{Ahv}. Similarly we estimate $J_3$ by
\begin{align*}
|J_3| \leq  \frac{C}{T^2}  \|h\|^2_{L^2_{\mu}}\|g\|^2_{L^2_{\mu}}.
\end{align*}
Collecting  the estimates finishes the proof.
\end{proof}

\section{ODE argument and proof of the main theorem}

In this section, we conclude with the proof of Theorem~\ref{thm:main}. First, we use a bootstrap argument and the results of the previous sections to show the smallness of the perturbation $f$ at a time $t^*>0$. We recall the existence Theorem~\ref{thm:existence}, which allows us to define, for $R>0$, the time:
\begin{align} \label{def:TR}
	T^*_R := \sup \left \{0<t\leq \frac12: \|g\|_{L^2_\M}\le R m T(t)   \text{ and } \|h\|_{L^2_{\mu_{T(t)}}}\leq R \delta^{\alpha-\frac14}\right\},
\end{align}
where $T(t) = \delta + 2c_0 t$ and $m = \delta ^\alpha$ with $\alpha \in (1/4, 1/2)$. 

\begin{lemma} \label{LemmaT*}
	There exists $R>0$ large enough and $\delta>0$ small enough such that the time $T^*_R$ defined in~\eqref{def:TR} satisfies $T^*_R=\frac12$. 
\end{lemma}
\begin{proof} We split the proof into two parts. First we control the growth of the far-field perturbation $g$, then the size of the near-field perturbation $h$.

\paragraph{ODE argument for $g$: }
For the estimate of the time derivative of $\|g(t)\|_{L^2_\M}^2$ we use the bounds as indicated below
\begin{align*}
\frac{1}{2} \partial_t \int_{\Reals^3} (g^2) \M^{-1} \;dv =&\underbrace{\int_{\Reals^3} Q(g,g)g\M^{-1} \;dv}_{\eqref{eq:gnonlin3}}  +  {\underbrace{\int_{\Reals^3} L_\M (g)  g\M^{-1} \;dv}_{\eqref{eq:LM}}} \\
&+ \underbrace{ \int_{\Reals^3} \Tr(A[h]\nabla^2\M)g\M^{-1} \;dv}_{\eqref{eq:AhM}}  + \underbrace{\int_{\Reals^3}  \Tr(A[h]\nabla^2 g) g\M^{-1} \;dv}_{\eqref{eq:gnonlin2}}\\
&+ \underbrace{\int_{\Reals^3} \Tr(A[E]\nabla^2g)g\M^{-1} \;dv}_{\eqref{eq:AEg}} +  {\underbrace{\int_{\Reals^3} \Tr(A[E]\nabla^2 \M) g\M^{-1} \;dv}_{\eqref{eq:SM}}}.
\end{align*}
This yields an upper bound for the time derivative of 
\begin{align*}
\frac{1}{2} \partial_t \| g\|^2_{L^2_{\M}} \leq  & C   \|g\|_{L^2_\M}\left( \|g\|^2_{L^2_\M}+ D_\M(g) \right)+ C \|g\|_{L^2_\M}^6 + \eps  D_\M(g) \\
& + 	\int_{\Reals^3} g L_\M(g) \M^{-1} \ud{v} + C T \|h\|_{L^2_T}  \|g\|_{L^2_\mu} \\
& + C T^\frac34 \|h\|_{L^2_\mu}\left({{ \|g\|^2_{L^2_\M}}}+ D_\M(g) \right)+  C T^3 \|g \|_{L^2_\M}^2 \|h\|_{L^2_\mu}^4+ \eps   D_\M(g) \\
& - D_{E}(g) + C \frac{m}{T} \|g\|^2_{L^2_\mu}  + C m \|g\|_{L^2_\mu}. 
\end{align*}

We use~\eqref{eq:coercImproved} to obtain 
\begin{align*}
	\frac{1}{2} \partial_t \| g\|^2_{L^2_{\M}} \leq & C  \|g\|_{L^2_\M}\left( \|g\|^2_{L^2_\M}+ D_\M(g) \right)+ C \|g\|_{L^2_\M}^6 + \eps  D_\M(g) \\
	& - D_\M(g) +C  T \|h\|_{L^2_T}  \|g\|_{L^2_\M} + C \|g\|_{L^2_\M}^2 \\
	& + C T^\frac34 \|h\|_{L^2_\mu}\left({{ \|g\|^2_{L^2_\M}}}+ D_\M(g) \right)+ C T^3 \|g \|_{L^2_\M}^2 \|h\|_{L^2_\mu}^4+ \eps   D_\M(g) \\
	& - D_{E}(g) + C \frac{m}{T} \|g\|^2_{L^2_\mu}  + C m \|g\|_{L^2_\mu} .
\end{align*}
For $t\leq T^*_R$ as defined in~\eqref{def:TR}, choosing $\delta>0$ small enough, we can estimate this by 
\begin{align*}
	\frac{1}{2} \partial_t \| g\|^2_{L^2_{\M}} \leq & C    \frac{m}{T} \|g\|^2_{L^2_\mu} +C m \| g\|_{L^2_{\M}}.
\end{align*}
For the norm of $g$ this implies
\begin{align*}
	\partial_t \| g\|_{L^2_{\M}} \leq  & C    \frac{m}{T} \|g\|_{L^2_\mu} + Cm    ,
\end{align*}
and for $t\leq T^*_R$ we can use the bound on $g$ to find:
\begin{align} \label{eq:gBootstrap} 
	\| g(t,\cdot)\|_{L^2_{\M}} \leq  & C   t Rm^2 + Cmt .
\end{align}

\paragraph{ODE argument for $h$: }

We apply the same strategy to the time derivative of the norm of $h$: 
\begin{align*}
\frac{1}{2}\int_{\Reals^3} (h^2)_t  \mu^{-1} \;dv = & \underbrace{\int_{\Reals^3} Q(h,h)h \mu^{-1} \;dv}_{\eqref{Qhh}}  + \underbrace{ \int_{\Reals^3} Q(g, h) h \mu^{-1} \;dv}_{\eqref{Q(g,h)}} +\; \underbrace{\int_{\Reals^3} L_E(h) h \mu^{-1} \;dv}_{\eqref{Q(E,h)+Q(h,E)}} \\
&\; + \underbrace{\int_{\Reals^3} Q(\M,h) h \mu^{-1} \;dv}_{\eqref{Q(M,h)}} 
+ \underbrace{\int_{\Reals^3} Q(g,E) h \mu^{-1} \;dv}_{\eqref{Q(g,E)}}  +  \underbrace{\int_{\Reals^3} gh^2 \mu^{-1} \;dv}_{ \eqref{gh} }\\
&\; + \int_{\Reals^3} \M h^2 \mu^{-1} \;dv + \underbrace{ \int_{\Reals^3} gE h \mu^{-1} \;dv}_{\eqref{come} }  + \underbrace{\int_{\Reals^3} S_E h \mu^{-1} \;dv}_{ \eqref{SE}}.
\end{align*}
Inserting the estimates yields 
\begin{align*}
\frac{1}{2} \frac{\operatorname{d}}{\operatorname{dt}}&\int_{\Reals^3} h^2  \mu^{-1}_{T(t)} \;dv 
 \leq  \; C( \varepsilon  +  T   \|h\|_{L^2_{\mu}} + \|g\|_{L^2_{\M}}) \int_{\Reals^3}  \frac{| \nabla h|^2}{\langle v \rangle ^3} \mu^{-1} \;dv  + T  \|h\|^4_{L^2_{\mu}}+ \|h\|^3_{L^2_\mu}  \\
 &+  \frac{1}{T^2}  \|h\|^2_{L^2_{\mu}}\|g\|^2_{L^2_{\M}}
 + \| g\|^4_{L^2_{\M}} \|h\|^2_{L^2_\mu} + \frac{1}{T^{3/2}} \| g\|_{L^2_{\M}}  \|h\|^2_{L^2_\mu}  - D_\mu(P^\perp_E h) - \int_{\Reals^3}  \frac{| \nabla h|^2}{\langle v \rangle^3}  \mu^{-1} \;dv \\
& +  \frac{m}{T^{\frac{9}{4}}}\|h \|_{L^2_{\mu}}\|g \|_{L^2_{\M}} +  {{\frac{m}{T^{\frac{5}{2}}}}}\|h \|_{L^2_{\mu}}\|g \|_{L^2_{\M}} \\
& + \| g\|^4_{L^2_{\M}} \|h\|^2_{L^2_\mu} + \frac{1}{T^{3/2}} \| g\|_{L^2_{\M}}  \|h\|^2_{L^2_\mu} +  \varepsilon  \int_{\Reals^3} \frac{ |\nabla h|^2}{ \langle v \rangle^3} \mu^{-1} \;dv \\
& + \|h \|^2_{L^2_{\mu}} + \frac{m}{T^{\frac{9}{4}}}\|h \|_{L^2_{\mu}}\|g \|_{L^2_{\M}}  +C m T^{-\frac12} \|h\|_{L^2_\mu}.
\end{align*}
Using the bootstrap assumption on $\|g\|_{L^2_\M}$ and $\|h\|_{L^2_\mu}$ yields, for $\delta>0$ small enough 
\begin{align*} 
\partial_t  \|h\|_{L^2_\mu} \leq C \left(     \frac{m}{T^\frac12} R \delta^{\frac12(2\alpha -\frac12)} + C {{\frac{\delta^{2\alpha}}{T^{\frac{3}{2}}}}}\right)   .
\end{align*}
Integrating the estimate, we find for $t\leq T^*_R$ that
\begin{align} \label{eq:hBootstrap} 
	\|h(t)\|_{L^2_\mu} \leq C m R \delta^{\frac12(2\alpha-\frac12)} + C \delta^{2\alpha- \frac12}.
\end{align}
Combining~\eqref{eq:gBootstrap} and \eqref{eq:hBootstrap}, we infer the existence of $R>0$ large enough and $\delta>0$ small enough such that $T^*_R = \frac12$ which proves the claim. 	
\end{proof}

With the previous lemma at hand, we are now in the position to finish the proof of the main theorem.

\begin{proof} [of Theorem~\ref{thm:main}] To conclude the proof of Theorem~\ref{thm:main} two things remain to show. The first is that at time $$t^* = \min \left \{ \frac{1}{2}, \frac{1}{4c_0} \right\},$$ the difference between $F$ and $\M$ is as small as we wish in the weighted $L^2_\M$ norm. First notice that 
$$
\| F-\M \|_{L^2_{\M}} \le  \| E  \|_{L^2_{\M}}  + \| g \|_{L^2_{\M}} +\| h \|_{L^2_{\M}}.
$$
The weighted norm of $g$, thanks to Lemma~\ref{LemmaT*},  is bounded at $t^*$ by 
\begin{align*}
\|g(t^*) \|_{L^2_\M} \leq C\delta^{\alpha}.
\end{align*}
Moreover, the norm of the local singularity $E$ behaves as  
\begin{align*}
\|E(t^*)\|^2_{L^2_\M} \leq C \frac{\delta^{2\alpha}}{(\delta + 2c_0t^*)^{3}} \int e^{-\frac{|v|^2}{\delta + 2c_0t^*}}e^{\frac{|v|^2}{2}}\;dv \;  \leq C  \; {\delta^{2\alpha}}.
\end{align*}
Finally, courtesy of Lemma~\ref{LemmaT*}, we have 
\begin{align*}
\delta^{2\alpha - \frac{1}{2}} \ge \| h(t^*) \|^2_{L^2_{\mu(t^*)}} \ge c \frac{1}{(\delta + 2c_0t^*)^{3/2}} \int h^2 e^{\frac{|v|^2}{2(\delta + 2c_0t^*)}}\;dv \ge c  \int h^2  e^{\frac{|v|^2}{2} }\;dv.
\end{align*}
Summarizing, 
$$
\| (E + f)(t^*) \|_{L^2_{\M}} \leq C  \delta^{\alpha - \frac{1}{4}}.
$$
Starting from the time $t^*$, the exponential decay of $F$ towards the equilibrium $\M_{eq}$ follows directly from Theorem 1.1 in \cite{carrapatoso2017landau}. It remains to check that $\| F-\M_{eq}\|_{L^2_\M}$ is small at time $t^*$. Finally, we note that 
$$
\|( F-\M_{eq})(t^*)\|_{L^2_\M} \le \| (F-\M)(t^*)\|_{L^2_\M} + \| \M-\M_{eq}\|_{L^2_\M} \leq C  \delta^{\alpha - \frac{1}{4}} + O(\delta),
$$
which concludes the proof.

\end{proof}
\nocite{*} 

\bibliographystyle{plain}
\bibliography{GW24}

\end{document}